\documentclass[12pt,twoside
]{article} 
\usepackage{amssymb,amsthm}
\usepackage[namelimits,sumlimits,fleqn]{amsmath}
\usepackage{setspace}
\usepackage{verbatim}
\usepackage{fixme}
\usepackage[hmargin={2.2cm,2.4cm}, top=36mm,
    a4paper,pdftex, textheight=230mm,
    headheight=20mm, headsep=10mm, bindingoffset=0.9cm]
    {geometry}
\usepackage[small, margin=20pt]{caption}
\usepackage[usenames,dvipsnames]{color}
\definecolor{darkblue}{rgb}{0.0,0.0,0.6}
\usepackage{float}\restylefloat{figure}
\usepackage{url}
\usepackage[backref=page]{hyperref}
\hypersetup{colorlinks=true,breaklinks=true,
            linkcolor=darkblue,urlcolor=darkblue,
            anchorcolor=darkblue,citecolor=darkblue}

\allowdisplaybreaks[4]

\theoremstyle{plain}
\newtheorem{thm}{Theorem}
\newtheorem{lem}[thm]{Lemma}
\newtheorem{cor}[thm]{Corollary}
\newtheorem{conj}[thm]{Conjecture}

\newtheorem{note}[thm]{Note}

\theoremstyle{definition}
\newtheorem*{definition}{Definition}

\newtheorem{question}{Question}

\theoremstyle{remark}
\newtheorem*{remark}{Remark}
\newtheorem*{remarks}{Remarks}
\newtheorem*{notation}{Notation}

\newtheoremstyle{case}{3pt}{3pt}{}{}{\bfseries}{:}{.5em}{}
\theoremstyle{case}
\newtheorem{case}{Case}

\newtheoremstyle{case2}{3pt}{3pt}{}{}{\bfseries}{:}{.5em}{}
\theoremstyle{case2}

\newtheoremstyle{kase}{3pt}{3pt}{}{}{\bfseries}{:}{.5em}{}
\theoremstyle{kase}
\newtheorem{kase}{Case}

\renewcommand*{\backref}[1]{}
\renewcommand*{\backrefalt}[4]{%
    \ifcase #1 (Not cited.)%
    \or        (p.\,#2)%
    \else      (pp.\,#2)%
    \fi}

\newcommand{\F}{\mathbb{F}} 
\newcommand{\R}{\mathbb{R}} 

\newcommand{\cI}{\mathcal{I}} 
\newcommand{\supp}{\mathrm{supp}} 
\newcommand{\ds}{\displaystyle} 
\newcommand{\f}{\frac} 
\newcommand\numberthis{\addtocounter{equation}{1}\tag{\theequation}} 

\newcommand{\E}{\mathsf{E}}
\renewcommand{\epsilon}{\varepsilon}

\begin{document}

\title{New results on sum-product type growth over fields}
\author{Brendan Murphy\footnote{Supported by the Leverhulme grant RPG 2017-371}, Giorgis Petridis\footnote{Supported by the NSF DMS Award 1723016, from the RTG in Algebraic Geometry, Algebra, and Number Theory at the University of Georgia, and by the NSF RTG grant DMS-1344994.}, Oliver Roche-Newton\footnote{Supported by the Austrian Science Fund FWF Project F5511-N26,
which is part of the Special Research Program "Quasi-Monte Carlo Methods: Theory and Applications" as well as Austrian Science Fund FWF Project P 30405-N32}\\
Misha Rudnev\footnote{Partially supported by the Leverhulme Trust Grant RPG 2017-371.}, and Ilya D. Shkredov\footnote{This work was supported in part by the Program of the Presidium of the
Russian Academy of Sciences 01 "Fundamental Mathematics and its Application" under grant PRAS-18-01.}}
\date{\today}
\maketitle

\begin{abstract} We prove a range of new sum-product type growth estimates over a general field $\F$, in particular the special case  $\F=\F_p$. They are unified by the theme of ``breaking the $3/2$ threshold'', epitomising the previous state of the art.  

This concerns two pivotal for the sum-product theory questions, which are lower bounds for the number of distinct cross-ratios determined by a finite subset of $\F$, as well as the number of values of the symplectic form  determined by a finite subset of $\F^2$.

We establish the estimate $|R[A]| \gtrsim |A|^{8/5}$  for cardinality of the set $R[A]$ of distinct cross-ratios, defined by triples of elements of a (sufficiently small if $\F$ has positive characteristic, similarly for the rest of the estimates) set $A\subset \F$, pinned at infinity. The cross-ratio bound enables us to break the threshold in the second question: for a non-collinear point set $P\subset \F^2$, the number of distinct values of the symplectic form $\omega$ on pairs of distinct points $u,u'$ of  $P$ is  $|\omega(P)|\gtrsim |P|^{2/3+c},$ with an explicit $c$. Symmetries of the cross-ratio underlie its local growth properties under both addition and multiplication, yielding an onset of growth of products of difference sets, which is another main result herein.  

Our proofs strongly use specially suited applications of new incidence bounds over $\F$, which apply to higher moments of representation functions. The technical thrust of the paper is using additive combinatorics to relate and adapt these higher moment bounds to growth estimates. A particular instance of this is breaking the threshold in the {\em few sums, many products} question over any $\F$, by showing that if $A$ is sufficiently small and has additive doubling constant $M$, then $|AA|\gtrsim  M^{-2}|A|^{14/9}$. This result has a second moment version, which allows for new upper bounds for the number of collinear point triples in the set $A\times A\subset \F^2$, the quantity often arising in applications of geometric incidence estimates.

\end{abstract}

\tableofcontents

\section{Introduction}
Let $\F$ be a field.
We assume throughout that $\F$ has positive characteristic $p$,  let $\F_p$ denote the corresponding prime residue field and $\F_q$  the field with $q$ elements and characteristic $p$. We view $p$ as a large asymptotic parameter. In the case of small $p$, the constraints in terms of $p$ we impose trivialise the results; these results extend to zero characteristic by removing the latter constraints and other terms containing $p$.


For a finite $A\subset\F$, we use the standard notation
\[
A\pm A = \{a\pm a':\,a,a'\in A\}
\]
for the sum or difference set of $A$, respectively; similarly we use $AA$ or $A/A$ for the product set or ratio set of $A$.
The cardinality $|A|$ is also viewed as an asymptotic parameter, however it has to be sufficiently small relative to $p>0$ for the questions below to be valid.

This paper studies a range of questions in arithmetic and geometric combinatorics, arising directly in the context of the  Erd\H{o}s-Szemer\'edi sum-product conjecture. The inter-relation of these questions accounts for its length. We first  outline heuristically these questions and make a brief case why they are important. We then present the formal list of questions and main results. They are paralleled to what is known in the special case of $\F$ being the real (or complex) field, implying that in many instances what we can prove for a general $\F$ is not much worse, making one believe that the keys to the sum-product conjecture may be independent of the order properties of $\mathbb R$.

A central problem in arithmetic combinatorics is the \emph{sum-product problem}, which asks for estimates of the form
\begin{equation}
  \label{eq:1}
  \max(|A+A|,|AA|)\geq C_\epsilon |A|^{1+\epsilon}
\end{equation}
for some $\epsilon>0$ and some constant $C_\epsilon$, independent of $A$.

This question was originally posed for finite sets of integers  \cite{ESz}: Erd\H os and Szemer\'edi conjectured that \eqref{eq:1} holds for all $\epsilon < 1$.
The sum-product problem has since been studied over a variety of fields and rings.
In  positive characteristic, the first qualitative estimate (i.e., without an explicit lower bound for $\epsilon$)  of the form \eqref{eq:1}  was  proved by Bourgain, Katz, and Tao \cite{BKT2004} for $\F=\F_p$. Today (for sufficiently small $A$ if $p>0$), the best known results are that $\epsilon$ can be taken as any constant smaller than $2/9$ for the general $\F$ \cite{RuShSh} and smaller than $1/3+5/5277$ for the special case $\F=\R$ \cite{RSS, Shakan2018}.

Until recently sum-product type estimates, dealing with sufficiently small sets in positive characteristic, mostly pertained to the case $\F=\F_p$,  based on the arithmetic approach founded in  \cite{BKT2004}. They were quite a way off the real case, where the main tools were the Szemer\'edi-Trotter theorem, see e.g. \cite{Elekes-Ruzsa2003}, and an order-based argument of Solymosi \cite{So2}. This changed recently after some geometric incidence theorems over the general $\F$ have been developed, based on the fourth author's point-plane incidence theorem \cite{Rudnev}. See \cite{murphy2017second} for a review.

Already qualitative sum-product estimates over $\F_p$ have had remarkable consequences as a fountainhead for the major theme of growth and pseudo-randomness throughout mathematics, as well as computer science and complexity theory. See, for example, a fairly recent review \cite{H} by Helfgott for growth in Lie type groups. Here we do not address further the issue of external applications of sum-product bounds.

\begin{notation}
Since this paper is about inequalities, let us describe notation to this end.
We use the standard notation $f = O(g)$, $f \ll g$, $g = \Omega(f)$ and $g \gg f $, meaning that there exists an absolute constant $C$ such that $f \leq C g$.
Both statements $f = \Theta(g)$ and $f \approx g$ stand for $f=O(g)$ and $f= \Omega(g)$.

In the inequalities in question, our primary concern is powers of set cardinalities, first and foremost of $|A|$. However, at {\em two} instances throughout the paper  our key estimates inadvertently acquire the term $\log|A|$: Theorems \ref{T(A)} and \ref{thm:eigenval} below. 
We decided not to keep track of the powers of $\log|A|$ throughout most of the ensuing estimates, but for writing $\log|A|$ explicitly in a few further estimates that then proliferate throughout, hoping that  this will be sufficient for an interested reader to easily calculate powers of logarithms elsewhere.
Notation-wise,  for positive quantities $f$ and $g$, by $f\lesssim g$ we mean that there is a constant $C >0$ such that $f\ll (\log g)^C g$ and the notation $f\gtrsim g$ means that $g\lesssim f$, or $f \gg (\log f)^{-C} g$. If both $f \lesssim g$ and $g \lesssim f$ hold, then we write $f \sim g$.
\end{notation}

\medskip
The heuristic behind the sum-product problem is that a subset $A$ of $\F$ cannot simultaneously have strong additive and multiplicative structure (unless $A$ is close to a sub-field). This suggests weaker formulations than \eqref{eq:1}. One is the {\em few sums, many products} question: is it true that if $|A+A|=M|A|$, then $|AA|\gtrsim M^{-C}|A|^2$, for some $C$? It was shown by Elekes and Ruzsa \cite{Elekes-Ruzsa2003} that this question over the real/complex field is within reach of applications of the Szemer\'edi-Trotter theorem; this is even true in the stronger $L^2$-sense \cite{RSS}. However, the known proofs of the  Szemer\'edi-Trotter theorem \cite{Szemeredi-Trotter1983}, \cite{To} do not extend beyond $\F=\mathbb C$. In the context of the few sums, many products over a general $\F$, we establish the best known bound in the forthcoming Theorem  \ref{thm:fsmp}.

However, the key sum-product issue is the {\em weak Erd\H os-Szemer\'edi conjecture}, or the {\em few products, many sums} question: is this true that if $|AA|=M|A|$, then $|A+A|\gtrsim M^{-C}|A|^2$, for some $C$? This central question is wide open over any field but the rationals, where it was resolved by Bourgain and Chang \cite{BC}, the proof relying strongly on integers being a Principal Ideal Domain. For some best known results  in the real category see \cite{S2}.

\medskip
It was shown \cite{IRR1} that the weak Erd\H os-Szemer\'edi conjecture would easily follow from a sharp $L^2$-bound on the number of realisations function of the quantity $\omega(u,u')$, where $u,u'\in A\times A$ and $\omega$ is the symplectic form; such a bound was erroneously claimed in \cite{IRR} in the wake of the resolution of the celebrated Erd\H os distinct distance conjecture by Guth and Katz \cite{GK}. Moreover, the  Erd\H os-type question of the minimum number of distinct values of a non-degenerate bilinear form on a plane non-collinear set $P$ now stands out as the major unsolved problem in plane discrete geometry, the forthcoming Conjecture \ref{cj}. Remarkably, unless the form is skew-symmetric, nothing better than a threshold lower bound $|P|^{2/3}$ is known for any $\F$. But if $\omega$ is the symplectic form and $\F= \mathbb R, \mathbb C$, there is a better bound \cite{IRR1}, \cite{Rudnev1}. 
However, the proofs in the latter two papers strongly rely on Szemer\'edi-Trotter type theorems.  It turns out that the ideas in these proofs can embrace slightly weaker incidence results.  The forthcoming Theorem \ref{arr} breaks the above threshold bound in the general $\F$ context.

\medskip
The sum-product phenomenon concerns the affine group action on $\F$. This insight was key to  the Cayley graph expansion theory founded by Helfgott, see \cite{H}. Not surprisingly, the set of invariants under affine group action, generated by a finite set $A$ and defined as $R[A]$ by \eqref{Rdef}, plays a vital role in the quantitative study of the sum-product phenomenon. The forthcoming Theorem \ref{R[A]} provides threshold-breaking lower bounds for cardinality of $R[A]$ in the general $\F$ context. These bounds become key to establishing the above-mentioned Theorem \ref{arr}. It is well known but nonetheless remarkable that in the special case $\F=\mathbb R, \mathbb C$ a near-sharp bound $|R[A]|\gtrsim |A|^2$ follows from the Szemer\'edi-Trotter theorem, essentially underlying the  {\em few sums-many products} argument of Elekes and Ruzsa \cite{Elekes-Ruzsa2003}. 

Thus the set  $R[A]$, over a general $\F$, is the true protagonist in this paper. Members of $R[A]$ can be viewed as cross-ratios defined by quadruples of pair-wise distinct elements of $A$, one of which elements has been pinned at infinity. The full cross-ratio is invariant under $SL_2$, acting on the projective line $\F P$ as M\"obius transformations. This is why the cross-ratio plays the central role in the proof of Theorem \ref{arr}: the symplectic form $\omega(u,u')$ is invariant to $SL_2$ action on $\F^2$ as linear transformations. Studying $SL_2$ action on $\F P$ therefore arises as another pivotal issue, closely related to the sum-product conjecture.

Lower bounds for cardinality of the set $C[A]$ of distinct cross-ratios generated by $A\subset \F$ constitute another unsolved central issue apropos of the sum-product phenomenon. The  bound $|C[A]|\gtrsim|A|^{2+2/11}$ over $\mathbb C$ was recently obtained by the fourth author \cite{Rudnev1} (where further discussion of the cross-ratio and its role throughout arithmetic growth may give a sceptical reader additional motivation to the scope of this paper). For general fields the best results we can so far prove for $C[A]$ are those for its subset $R[A]$.

\subsection{Key questions and theorems}

We start out with a question not yet mentioned, yet also directly dependent on the symmetry properties of the set $R[A]$, as well as its cardinality.

\begin{question}[Difference sets are not multiplicatively closed]
  \label{q1}
Given a finite subset $A$ of a field, let $D=A-A$ denote its difference set.
Do we have an estimate of the form
\begin{equation}
  \label{eq:2}
  |DD|\geq |D|^{1+\epsilon}
\end{equation}
for some $\epsilon>0$?
\end{question}
When $A$ is a subset of $\R$, fifth listed author \cite{S_diff} proved
\[
|DD| \gtrsim {|D|^{1+\frac 1{12}}}.
\]
See \cite{Rudnev1} for additional discussion as to taking more products of $D$ with itself. We claim a  bound in the form \eqref{eq:2} over a general $\F$.
\begin{thm} \label{thm:proddiff} Let $A \subset \mathbb F$, with $|A| \leq p^{5/12}$ and $|A|^{16}|A-A|^{30} \leq p^{25}$. Then
\[
|(A-A)(A-A)|  \gtrsim |A-A|^{4/5} |A|^{\frac{32}{75}} \geq  |A-A|^{1+\frac{1}{75}} .
\]
\end{thm}

Note that these two conditions are satisfied for sufficiently small sets, and certainly hold if $|A| < p^{25/76}$. The first inequality of Theorem \ref{thm:proddiff} implies that for sufficiently small $A\subset \mathbb F_p$ 
\[
 |(A-A)(A-A)| \ll |A|^2 \Rightarrow |A-A| \lesssim |A|^{2-\frac{1}{30}}.
\]
This is the first partial result in the finite field setting towards a conjecture in \cite{RNZ}. See also \cite{SZH} for related results.

More generally, it was conjectured in \cite{BRZ} that over the reals, for any $n$ there is a sufficiently large $m$, so that
\[
|D^m| \geq |D|^n.
\] 
That is, the difference set expands ad infinitum under multiplication.
It was shown in \cite{BRZ} that
\[
|D^3| \gtrsim |A|^{17/8}.
\]
We are also able to prove something similar here in the finite field setting, giving a six-variable expander with superquadratic growth.

\begin{thm} \label{thm:superquadratic} Let $A \subset \F_p$ such that $|A|\leq p^{25/77}$ and write $D=A-A$. Then
\[|DD/D| \gtrsim |A|^{2+\frac{1}{25}}
\]

\end{thm}

The proofs of Theorem \ref{thm:superquadratic} and Theorem \ref{thm:proddiff} are similar. In fact, it is possible to derive a quantitatively weaker but still non-trivial version of Theorem \ref{thm:proddiff} as a consequence of Theorem \ref{thm:superquadratic}, via an extra application of the Pl\"{u}nnecke-Ruzsa inequality.

\begin{question}[Few sums, many products]
  \label{q2}
If $|A+A|\leq |A|^{1+\epsilon}$ for a small $\epsilon>0$, is $|AA|\geq |A|^{2-\delta}$ for $\delta=\delta(\epsilon)>0$?
If so, can we take $\delta\to 0$ as $\epsilon\to 0$?
\end{question}
As we have mentioned, this is a weaker version of the sum-product conjecture of Erd\H{o}s and Szemer\'edi.
Over $\R$, Question~\ref{q2} has been completely resolved by Elekes
and Ruzsa \cite{Elekes-Ruzsa2003}, who showed that if $|A+A|\leq |A|^{1+\epsilon}$, then
\[
|AA|\gg \frac{|A|^{2-4\epsilon}}{\log|A|}.
\]
In \cite{RSS} $4\epsilon$ gets replaced by $3\epsilon$, and the formulation is true involving $A-A$ as well as in the stronger $L^2$-sense.  Solymosi's well-known result \cite{So2} replaces in the above inequality $4\epsilon$ with $2\epsilon$ but is specific to $A+A$.

In this vein \cite{AYMRS}, it was shown that if $A$ is a subset of a field of characteristic $p$ satisfying $|A|<p^{3/5}$ and $|A+A|\leq |A|^{1+\epsilon}$ then
\[
|AA|\gg |A|^{\frac 32(1-\epsilon)}.
\]
Here we prove the following theorem.
\begin{thm} \label{thm:fsmp} Let $A\subset \mathbb F$ and $\lambda \in \F^*$. 
Then if $|A|\leq p^{5/9}$
\begin{equation}
|A+ \lambda A|^{21}|A/A|^{10} \gtrsim |A|^{37}
\label{fsmp1}
\end{equation}
and if $|A|\leq p^{9/16}$
\begin{equation}
|A+ \lambda A|^{18}|AA|^9 \gtrsim |A|^{32}
\label{fsmp2}
\end{equation}
In particular, if $|A+\lambda A| = M|A|$ then $|A/A| \gtrsim M^{-21/10} |A|^{8/5}$ and $|AA| \gtrsim M^{-2} |A|^{14/9}$.
\end{thm}
The most natural instances of Theorem \ref{thm:fsmp} occur by taking $\lambda = \pm 1$, in which case  $|A\pm A|\leq |A|^{1+\epsilon}$ implies
\[
|AA| \gtrsim |A|^{\frac{14}9 -2\epsilon}\qquad\mbox{and}\qquad |A/A|\gtrsim |A|^{\frac 85 - \frac{21\epsilon}{10}}.
\]
Both exponents $14/9$ and $8/5$ exceed $3/2$, which was a ``natural'' threshold appearing in many earlier estimates \cite{murphy2017second}.

The third question we consider is geometric.
 Let $\omega$ be a non-degenerate skew-symmetric bilinear form on $\F^2$. Since the vector space of such forms is one-dimensional, without loss of generality $\omega$ is the standard symplectic (or area) form
\[
\omega[ (u_1,u_2), \,(u_1',u_2') ] = u_1u_2'-u_2u_1'.
\]
For a point set $P\subset \F^2\setminus\{0\}$, what is the minimum cardinality of
\begin{equation}
\omega(P) :=\{\omega(u,u'):\,u,u'\in P\}\setminus\{0\}?
\label{qu}\end{equation}
In the exceptional case where $P$ is supported on a single line through the origin, then $\omega(P)$ can be empty.
If this is not the case, one can make the following conjecture.
\begin{conj}
\label{cj}
Suppose $|P|\leq p$.
Either $\omega(P)$ is empty or
\begin{equation}
  \label{eq:41}
  |\omega(P)|\geq |P|^{1-\epsilon}
\end{equation}
for all $\epsilon > 0$.
\end{conj}
The estimate  $|\omega(P)|\gg \min\{|P|^{2/3},p\}$ holds over any field $\F$, see \cite[Section 6.1]{Rudnev}, which establishes \eqref{eq:41} for $\epsilon=1/3$.
This implies that if $|P|\gg p^{3/2}$, then $|\omega(P)|\gg p$; again, the exponent $3/2$ appears as a natural threshold.
It is difficult to surpass this threshold even in characteristic zero, where the best known bound \cite{IRR1} is $|\omega(P)|\gg|P|^{96/137}$.

We break the  threshold now in positive characteristic, which provides some evidence towards Conjecture~\ref{cj}.
\begin{thm}
\label{arr}
Let $\omega$ be a non-degenerate skew-symmetric bilinear form on $\F^2$.
If $P\subset \F^2$ is not supported on a single line through the origin and $|P|\leq p^{161/162}$, then
\[
|\omega(P)|\gtrsim |P|^{108/161}.
\]
\end{thm}

\begin{remarks}
If the form $\omega$ is symmetric, rather than skew-symmetric, nothing better than $|\omega(P)|\gg |P|^{2/3}$ (with $|P|\leq p^{3/2}$ if $p>0$) is known for any field $\F$, including the rationals; the latter estimate over the reals arises as a one-step application of the Szemer\'edi-Trotter theorem.

Conjecture~\ref{cj} appears at the first glance to be similar to the renown Erd\H{o}s distinct distance problem, which asked for the minimum size of the set $\Delta(P)$ of pairwise distances between points of $P$.
Erd\H{o}s conjectured that $|\Delta(P)| \gg |P|/\sqrt{\log |P|}$.
The conjecture was resolved (up to a power of $\log |P|$, which was at the mercy of the Cauchy-Schwarz inequality) by Guth and Katz ~\cite{GK} who showed that
$$
  |\Delta(P)| \gg \frac{|P|}{\log |P|}.
$$
An attempt to mimic the Guth-Katz approach to the Erd\H{o}s distance problem to deal with Conjecture~\ref{cj} fails: see \cite{IRR}, \cite{IRR1}.
Today Conjecture~\ref{cj} is a major open question in geometric combinatorics.
\end{remarks}

\subsection{Key estimates}
The following results form the backbone of the paper. In particular, they contain the estimates of cardinality the above-mentioned set $R[A]$, underpinning the above Theorems \ref{thm:fsmp} and \ref{arr}.

\subsubsection*{Estimates for collinear triples and quadruples}
The first pair of key estimates bound the number $T(A)$ of \emph{collinear triples}\/ and the number $Q(A)$ of \emph{collinear quadruples} of $A$.
Geometrically, $T(A)$ and $Q(A)$ are the number of ordered collinear triples and quadruples in the point set $A\times A\subset \F^2$. They represent respectively the third and fourth moment of the representation function $i(l)$, where $l$ is a line defined by a pair of points in $A\times A$.

Algebraically, $T(A)$ is the number of solutions to
\[
(a_1 -a_2)(a_3-a_4)=(a_1-a_5)(a_3-a_6)
\]
with all the $a_i \in A$, and $Q(A)$ is the number of solutions to
\[
(a_1 -a_2)/(a_3-a_4)=(a_1-a_5)/(a_3-a_6) = (a_1-a_7)/(a_3-a_8)
\]
with all the $a_i \in A$.
In the formula above we mean that if a denominator is zero, then all others denominators are zeros as well. 
If $A$ is chosen uniformly at random from $\F_p$, then we expect $T(A)\approx |A|^6/p$ and $Q(A)\approx |A|^8/p^2$.
In Section~\ref{sec:incidence_bounds} we show that the expected values hold for large subsets $A\subset \F_p$.
In particular, we show that
\[
\left\vert Q(A)-\frac{|A|^8}{p^2} \right\vert \ll |A|^5\log |A|,
\]
which is sharp (up to the logarithmic factor).
These results, which we state formally as Theorem~\ref{T(A)} in Section~\ref{sub_section_collinear}, are based on an incidence bound of Stevens and de Zeeuw \cite{SdZ}, which adapts the bound of Rudnev \cite{Rudnev} for point-plane incidences.
Theorem~\ref{T(A)} and the results leading to it will be used throughout the paper.

\subsubsection*{Expansion properties for $R[A]$}
The second key estimate is a growth result for the set $R[A]$.

Let  $A \subset \F$. Set 
\begin{equation}
R[A] = \left\{ \frac{b-a}{c-a} : \, a,b,c \in A;\,b,c\neq a  \right\}.
\label{Rdef}\end{equation}
This set was studied by Jones \cite{Jones2013} in the real setting, where it was proven that $|R[A]| \gg |A|^2 / \log |A|$. Geometrically, $R[A]$ is the set of ``pinned'' cross-ratios for the set $A\cup\{\infty\}$ on the projective line in the form $[a,b,c,\infty]$, where $a,b,c\in A$. Recall that the cross-ratio of four pair-wise distinct points $a,b,c,d \in \F$ is defined as
\[
[a,b,c,d]:=\frac{(a-b)(c-d)}{(a-c)(b-d)}.
\]
For the set $R[A]$, the last variable is fixed at infinity.

The following result shows that any subset of $\F$ determines a large number of pinned cross-ratios.
\begin{thm}
\label{R[A]}
Let $A \subset \F$. The following expansion properties hold for $R[A]$.
\begin{enumerate}
\item (Expansion for small sets) If $|A| \leq p^{5/12}$, then $\ds |R[A]| \gg \frac{|A|^{8/5}}{\log^{8/15}|A|}.$
\item (Expansion for large sets) If $\F=\F_p,$ then $|R[A]| \gg \min\left\{ p, \frac{|A|^{5/2}}{p^{1/2}} \right\}.$ In particular, if $|A| \geq p^{3/5}$, then $\ds |R[A]| \gg p$. 
\item (Uniform expansion) If $\F=\F_p,$ then $|R[A]| \gg \min\{p, |A|^{\frac{3}{2} + \frac{1}{22}}\log^{-\frac 49} |A|\}$.
\end{enumerate}
\end{thm}
The bound
\[
|R[A]| \gtrsim \min\{p, |A|^{\frac{3}{2} + \frac{1}{22} } \}.
\]
improves the lower bound $|R[A]| \gg \min\{p, |A|^{\frac{3}{2}}\}$ established in \cite{AYMRS}.
Theorem~\ref{R[A]} underlies all of the results in Section~\ref{sec:RAapp}.

An alternative interpretation of Theorem \ref{R[A]} is that it offers strong expansion for the rational function $R(x,y,z) = \tfrac{y-x}{z-x}$. 
Pham, Vinh, and de Zeeuw ~\cite{phamQuad} showed that a large class of quadratic polynomials $f$ in three variables satisfy the expansion bound $|f(A)| \gg \min\{p, |A|^{3/2}\}$.
Theorem~\ref{R[A]} offers stronger expansion for $R[A]$, though $R$ is a rational function in three variables.
Such expanding polynomials have applications to theoretical computer science \cite{barak2006extracting,zuckerman1990general,zuckerman1991simulating}.
See \cite{phamQuad} for further discussion.

\subsubsection*{Notation}
By $\F^*$ we mean the multiplicative group  of $\F$, by $\F^d$ the $d$-dimensional vector space over $\F$, and by $\F \mathbb{P}^d$ the $d$-dimensional projective space over $\F$. 

For $x,y \in \F$ we write $\tfrac{x}{y}$ instead of $xy^{-1}$ and implicitly assume $y \neq 0$. We use representation function notations like $r_{RR}(x)$, which counts the number of ways $x \in \F$ can be expressed as a product $r_1 r_2$ with $r_1, r_2$ in some finite set  $R$. 

Given a family $L$ of lines in $\F^2$, $\cI(A \times A, L)$ is the number of point-line incidences between $A \times A$ and $L$; that is, the number of ordered pairs $(u, \ell) \in (A \times A) \times L$ such that $u \in \ell$.

We denote the additive energy of finite sets $A,B$ in an abelian group $(G,+)$ by
\[
\E(A,B) := |\{(a,a',b,b')\in A\times A\times B\times B:\, a+b=a'+b'\}|,
\]
and write $\E(A)$ when $A=B$.
We have the trivial bounds
\[
|A||B|\leq \E(A,B) \leq |A|^2|B|, |A||B|^2.
\] 

By the Cauchy-Schwarz inequality, we have the lower bounds
\begin{equation}
|A+A|,|A-A|\,\geq\, \frac{|A|^4}{\E(A)}.
\label{csest}\end{equation}
Thus if $\E(A)=O(|A|^2)$, then $|A+A|\gg |A|^2.$

Note that $A+A$ is the support of the convolution $A\ast A$ of the characteristic function of $A$ with itself and $\E(A)$ is the second moment of $A\ast A$.
To this end, we will also use the third moment of convolution.
We define the \emph{third moment}, alias third, or cubic, energy of finite sets $A,B$ in an abelian group by
\[
\E_3(A,B) = | \{(a_1,a_2,a_3,b_1,b_2,b_3)\in A^3 \times B^3:\, a_1-b_1=a_2-b_2=a_3-b_3\}| = \sum_x r^3_{A-B} (x),
\]
with the shorthand $\E_3(A)=\E_3(A,A)$.
We have the trivial bounds $|A|^3\leq \E_3(A)\leq |A|^4$.

By H\"older's inequality, there is a lower bound for $|A-A|$ in terms of $\E_3(A)$ analogous to \eqref{csest}. However, if $\E_3(A)=O(|A|^3)$ this lower bound only implies that $|A-A|\geq |A|^{3/2}$.
See \cite{SS_moments} for more on higher moments of convolution.

Since we are working over a field, we use the notations $\E^+$ and $\E^\times$ to distinguish between additive and  multiplicative energy.

\subsection{Other results}

The paper contains a variety of additional results. We summarise a few more in order to help the reader locate them. Precise statements can be found in the relevant sections.
\begin{itemize}
\item Asymptotic expressions for the number $T(A)$ of ordered collinear triples determined by $A \times A$ when $\F = \F_p$ (Theorem~\ref{T(A)} on p.~\pageref{T(A)}):
\[
T(A) = \frac{|A|^{6}}{p} + O(p^{1/2}|A|^{7/2}).
\]
\item We break the $3/2$ threshold for products of differences and show, for example, that if $A \subset \F_p$, then for all $\varepsilon >0$
\begin{equation}
|(A \pm A)(A \pm A)| \gg \min\{ p, |A|^{\frac{3}{2}+ \frac{1}{90} - \varepsilon} \}.
\label{1/90}
\end{equation}
See Theorem~\ref{PD} on p.~\pageref{PD} and the remark after the proof. 

\item We improve results of Hart, Iosevich, and Solymosi~\cite{HIS2007} and Balog~\cite{Balog2013} on four-fold product of differences in prime-order finite fields. We show that if $A \subset \F_p$, then
\begin{equation}
(A \pm A)(A \pm A)(A \pm A)(A \pm A) = \F_p
\label{4prod}
\end{equation}
provided that $|A| \gg p^{3/5}$. See Theorem~\ref{4-fold} on p.~\pageref{4-fold} and the remark after the proof. The previous best known lower bound was $|A| \gg p^{5/8}$, due to Hart, Iosevich, and Solymosi~\cite{HIS2007} who used Weil type Kloosterman sum bounds.
\item We obtain an improved upper bound on the number of ordered collinear triples when the set $A$ has small additive doubling. That is, if $|A + A| = M|A|$ and $A$ is sufficiently small, then 
\[
T(A) \lesssim M^{\frac{51}{26}}|A|^{\frac 92 - \frac 1{26}}.
\]
See Corollary~\ref{cor:T_first} on p.~\pageref{cor:T_first} and Corollary~\ref{cor:T_second} for a similar result when $A$ has small multiplicative doubling constant. (Namely that if $|AA|=M|A|$, then the above bound holds for $T(A)$.)
\item We also show that multiplicative subgroups of $\F_p$ may contain the difference set of only small sets. For example, we prove that if $\Gamma \subset \F_p$ is a sufficiently small multiplicative subgroup and $A$ is a set such that $A-A \subset \Gamma$, then $|A| \lesssim |\Gamma|^{5/12}$. This improves a result from \cite{S_diff}. See Theorem~\ref{Mult sbgp} on p.~\pageref{Mult sbgp}.
\end{itemize}

\subsection{Structure of the rest of the paper}
The rest of the paper is taken up by proofs of the above formulated main results and some corollaries. Section \ref{Giorgis}, with a self-explanatory title, is dedicated to the proof of Theorem \ref{R[A]}; it is in this section where our results on collinear triples and quadruples can be found. The bounds in this section depend only on the Stevens-de Zeeuw incidence bound.

Section \ref{Brendan} uses Theorem \ref{R[A]} to prove Theorem \ref{thm:proddiff}.

Section \ref{Ilya} uses Theorem \ref{R[A]} to establish a new bound on the size of a set $A\subset \F_p$, such that the nonzero part of the difference set $A-A$ is contained in a multiplicative subgroup of $\F_p^*$.

Section \ref{Misha} is exclusively dedicated to proving  a slightly cheaper version of Theorem \ref{arr}. It depends crucially on the bound in clause 1, Theorem \ref{R[A]} and we have chosen to present it this way in order to emphasise the role played by the cross-ratio. The full and more technical proof of Theorem \ref{arr} is presented in Appendix C.

Section \ref{Olly} uses incidence bounds and facts about higher energies to prove Theorem \ref{thm:fsmp}. This result is then used to give new bounds for $(A-A)(A-A)$, of the form $|(A-A)(A-A)| \gg |A|^{3/2+c}$ for some absolute $c>0$. Variations with more products of the difference set are also considered. Section \ref{Olly} also contains an energy variant of Theorem \ref{thm:fsmp}. In particular, if a set $A$ has small additive doubling, we prove a better bound on the number of collinear point triples in the point set $A\times A$ than stated by \eqref{T(A) known}.

\begin{remark}We would like to point out that as far as the set $\frac{A-A}{A-A}\subseteq \F_p$ is concerned, a much stronger, optimal up to a constant bound 
\begin{equation} \label{szonyi}
\left|\frac{A-A}{A-A}\right|\geq \frac{|A|^2+3}{2}, \;\;\;\mbox{for}\;\;\;|A|\leq\sqrt{p}
\end{equation}
has been known since at least  as early as the late 1990s. See, e.g., a review \cite{Szo} by T. Sz\H onyi and the references contained therein. This is due to the fact that $\frac{A-A}{A-A}$ (on the projective line) is the set of directions, determined by the point set $A\times A\subseteq \F_p^2$. For a non-collinear point set $P\subset \F_p^2$, with $|P|\leq p$, \cite[Theorem 5.2]{Szo}  claims the general optimal lower bound $\frac{|P|+3}{2}$ for the number of distinct directions, determined by pairs of points of $P$. It is easy to see that if $|P|>p$, then all $p+1$ directions are determined.

We will use \eqref{szonyi} in the proof of Theorem \ref{thm:superquadratic}. One may think of Theorem \ref{thm:superquadratic} as an extension of the above bound with additional variables.

The proof of \cite[Theorem 5.2]{Szo} belongs to the realm of what today is generally called {\em the polynomial method}. The polynomial method, in fact, also underlies the incidence results used in this paper, stemming from Theorem \ref{Rudnev}.\end{remark}

Finally, Appendices A and B present, respectively, additional discussion and the proof of a  somewhat technical Lemma \ref{soly2}.

\section[Proof of Theorem 5]{Proof of Theorem \ref{R[A]}} \label{Giorgis}

In this section we prove the key estimates alluded to in the introduction.

Owing to the statements 2 and 3 of Theorem~\ref{R[A]}, we assume throughout this section that $\F=\F_p$. A reader can easily verify that this assumption is not used in the proof of statement 1 of Theorem~\ref{R[A]}.

In Section~\ref{sec:incidence_bounds}, we state two incidences bounds needed throughout the paper, as well as a lemma that we will use to prove a bound for rich lines, Lemma~\ref{L_M}.

Section~\ref{sub_section_collinear} introduces collinear triples $T(A)$ and quadruples $Q(A)$. The main result of this section is Theorem~\ref{T(A)} on asymptotic formulas for $T(A)$ and $Q(A)$.

The remaining next two sections are dedicated to the proof of Theorem~\ref{R[A]}.

Appendix~\ref{sec:discussTAQA} contains further discussion of some of the remarks in Section~\ref{sub_section_collinear}.

\subsection{Incidence bounds}
\label{sec:incidence_bounds}
We will use the following version of the fourth listed author's point-plane incidence bound \cite{Rudnev} that is specialized to the $\F_p$-setting and the  case of equal number of points and planes. See \cite[Corollary 2]{murphy2017second} for a proof.

\begin{thm}\label{Rudnev}
Let $p$ be an odd prime,  $P \subset \F_p^3$ and a collection $\Pi$ of planes in $\F_p^3$. Suppose that $|P| = |\Pi|$ and that $k$ is the maximum number of collinear points in $P$. The number of point-plane incidences satisfies
\[
\cI(P, \Pi) \ll \frac{|P|^2}{p} + |P|^{3/2} + k |P|.
\]
\end{thm}
  
We will also need bounds on point-line incidences in $\F_p^2$.  For a line $\ell$ and a point set $P$ in the plane $\F_p^2$, $i(\ell)$ represents the number of points in $P$ incident to $\ell.$  In particular, for $P=A\times A$,
\begin{equation}\label{i function}
i(\ell) = |\ell \cap (A \times A)|.
\end{equation}
 The number of point-line incidences is defined by 
\[
\cI(A \times A, L) = \sum_{\ell \in L} i(\ell).
\]

The state of the art on point-line incidences is due to Stevens and de Zeeuw~\cite{SdZ}.
See \cite[Theorem 7]{murphy2017second} for the following version.

\begin{thm} \label{SdZgen}
Let $A,B \subset \F$ with $|A| \leq |B|$ and let $L$ be a collection of lines in $\F^2$.

If $\mathbb F$ has positive characteristic $p$, assume that
\begin{equation}
|A||L| \leq p^2.
\label{cond2}
\end{equation}
Then the number of incidences
$\cI(P,L)$ between the point set $P=A \times B$ and $L$ is bounded by
\[
\cI(P, L) \ll |A|^{3/4}|B|^{1/2}|L|^{3/4} +|P|+|L|.
\]
If we omit \eqref{cond2}, then the following estimate holds for $A,B\subset \F_p$
\[
\cI(P,L) \ll \frac{|A||B|^{1/2}|L|}{p^{1/2}}+|A|^{3/4}|B|^{1/2}|L|^{3/4} +|P|^{2/3}|L|^{2/3}+|L|.
\]
\end{thm}

It was shown in \cite{Rudnev} that Theorem \ref{Rudnev} generally cannot be improved. This is also the case with the Stevens-de-Zeeuw theorem, modulo factors of $\log|A|$. 

The following lemma is based on a refinement of~\cite[Lemma~2.1]{BKT2004}; see also~\cite{Konyagin2003,BGK2006}. We reformulate it in terms of the incidence function $i$ defined in~\eqref{i function}, c.f.~\cite[Lemma 1]{MPInc}. Sums are over all lines in $\F_{p}^2$ and not just those incident to some point of $A \times A$.
\begin{lem}\label{BKT i}
Let $A \subset \F_{p}$.
\[
\sum_{\text{ all lines } \ell}  i(\ell)^2 = |A|^4 + p |A|^2.
\]
In particular
\[
\sum_{\text{ all lines } \ell} \left( i(\ell) - \frac{|A|^2}{p} \right)^2 \leq p |A|^2.
\]
\end{lem}

This simple yet powerful lemma was implicitly extended to general point sets $P\subset \F_q^2$ by Vinh~\cite{Vinh2011}. The paper~\cite{MPInc} contains a proof and many applications.

Now we formulate the main lemma needed for the proof of Theorem~\ref{T(A)}.
Let $M$ be a parameter and set
\begin{equation}\label{LM}
L_M = \{ \ell \in L :  M <  i(\ell) \leq  2M \}
\end{equation}
to be the collection of lines from $L$ that are incident to between $M$ and $2M$ points in $A \times A$. 

\begin{lem}\label{L_M}
Let $A \subset \F_{p}$ and 
let $\ds 2|A|^2/p\leq M \leq |A| $ be an integer that is greater than 1.  The set $L_M$ defined in~\eqref{LM} satisfies
\[
|L_M| \ll \min\left(\frac{p|A|^2}{M^2},\frac{|A|^5}{M^4}\right).
\]
In particular, if $M\geq 2|A|^2/p > 1$ then
\[ 
\sum_{\ell \in L_M} i(\ell)^3 \ll \frac{|A|^5}{M} \text{ and } \sum_{\ell \in L_M} i(\ell)^4  \ll |A|^5.
\]
\end{lem}
The bound $|A|^5/M^4$ is superior when $M \geq |A|^{3/2}/p^{1/2}$.
\begin{proof}
The second part of the lemma follows by the first claim because, say, 
\[
\sum_{\ell \in L_M} i(\ell)^3 \leq 8 M^3 |L_M|.
\]

First we show that for all $M\geq 2|A|^2/p > 1$
\begin{equation}
\label{BKT In}
|L_M| \leq\frac{4p|A|^2}{M^2}.
\end{equation}
The hypothesis $\ds i(\ell) \geq \tfrac{2 |A|^2}{p}$ implies that $\ds i(\ell) - \tfrac{|A|^2}{p} \geq \tfrac{i(\ell)}{2}\geq \tfrac{M}{2}$.

Equation~\eqref{BKT i} now implies
\[
\frac{M^2}{4} |L_M| \leq \sum_{\ell \in L_M} \left( i(\ell) - \frac{|A|^2}{p} \right)^2  \leq \sum_{\text{ all lines } \ell} \left( i(\ell) - \frac{|A|^2}{p} \right)^2 \leq p |A|^2.
\]
Thus \eqref{BKT In} follows.

To show that $|L_M|\ll |A|^5/M^4$ for $M\geq 2|A|^2/p$, we consider two cases.
First, suppose that $M < \frac{2|A|^{3/2}}{p^{1/2}}$.
By Lemma~\ref{BKT In}, we have
\[
|L_M| \ll \frac{p|A|^2}{M^2} < \frac{p|A|^2}{M^2}\frac{4|A|^3}{pM^2} = \frac{4|A|^5}{M^4}.
\]

Next, suppose that $M \geq 2|A|^{3/2}/p^{1/2}$; in this case,
we apply Theorem~\ref{SdZgen}. Therefore we must confirm the condition $|A| |L_M| \leq p^2$. The hypothesis $M \geq 2 |A|^{3/2} / p^{1/2}$ implies $ M \geq 2 |A|^2 /p$. Lemma~\ref{BKT In} can therefore be applied and in conjunction with the hypothesis $M^2 \geq 4 |A|^3 / p$ gives
\[
|L_M| \leq \frac{4 p |A|^2}{M^2} \leq \frac{p^2}{|A|}.
\]
Hence, $|A| |L_M| \leq p^2$ and Theorem~\ref{SdZgen} may be applied. It gives
\[
M |L_M| \leq \sum_{\ell \in L_M} i(\ell) =  \cI(A \times A , L_M) \ll |A|^{5/4} |L_M|^{3/4} +  |A|^2 + |L_M|.
\]

If $2\leq M\leq 2C$, where $C$ is the implicit constant in Theorem~\ref{SdZgen}, then \[|L_M|\leq |A|^4 \ll |A|^5/M^4.\]

On the other hand, if $M > 2C$, it follows that
\[
M |L_M| \ll |A|^2 + |A|^{5/4} |L_M|^{3/4}.
\]
Hence, using $M \leq |A|$,
\[
|L_M| \ll \frac{|A|^5}{M^4} + \frac{|A|^2}{M} \ll \frac{|A|^5}{M^4}.
\]
\end{proof}

Note that conditions $1 < M \leq |A|$ are not restrictive; the only condition of significance in the statement of Lemma~\ref{L_M} is $M\geq 2|A|^2/p$.

\subsection{Collinear triples and quadruples}
\label{sub_section_collinear}

Many of the facts stated in this subsection are justified in detail in Appendix~\ref{sec:discussTAQA}. Let us first define ordered triples and quadruples of a Cartesian product point set.

\begin{definition}
Let $A \subset \F$. $T(A)$ is the number of ordered \emph{collinear triples} in $A \times A$. That is ordered triples $(u,v,w) \in (A \times A) \times  (A \times A) \times  (A \times A)$ such that $u$, $v$ and $w$ are all incident to the same line. $Q(A)$ is the number of ordered \emph{collinear quadruples} in $A \times A$.
\end{definition}

In this section, we combine Theorem~\ref{SdZgen} with a generalisation of a lemma of Bourgain, Katz, and Tao~\cite{BKT2004}, to prove an improved bound for $T(A)$ for large sets (here we use the assumption that $\F=\F_p$); and an optimal bound for $Q(A)$ (here our result is stated and proved in $\F_p$, but the argument works with small modifications in any field $\F$; if the characteristic of the field is zero then the $|A|^8/p^2$ term disappears; if $\F$ has positive characteristic, then the $|A|^8/p^2$ disappears provided that $|A|$ is small enough in terms of the characteristic $p$).

We begin with an formula for $T(A)$ and $Q(A)$ in terms of the incidence function defined in~\eqref{i function}.
\begin{equation}\label{Ti Qi}
T(A) = \sum_{\ell : i(\ell) >1} i(\ell)^3 + O(|A|^4) \quad \text{ and } \quad Q(A) = \sum_{\ell : i(\ell) >1} i(\ell)^4 + O(|A|^4).
\end{equation}

Alternative expressions for $T(A)$ and $Q(A)$ involve ratios of differences of elements of $A$. To this end we define two functions
\begin{equation} \label{t function}
t(x) = |\{ (a,b,c) \in A \times A \times A  \colon x = \tfrac{b-a}{c-a} \}| 
\end{equation}
and
\begin{equation} \label{q function}
  q(x,y)=|\{(a,b,c,d)\in A \times A \times A \times A \colon \tfrac{b-a}{c-a} = x \, , \, \tfrac{d-a}{c-a} = y \}|.
\end{equation}
We have the following two identities
\begin{equation} \label{t expression}
T(A) = \sum_{x \in \F_p} t(x)^2 + |A|^4 
\end{equation}
and 
\begin{equation} \label{q expression}
Q(A) = \sum_{x,y \in \F_p} q(x,y)^2 + |A|^5.
\end{equation}

Finally, there are formulas for $T(A)$ and $Q(A)$ in terms of the multiplicative energies of additive shifts of $A$. 
\begin{equation}\label{f:T_via_E}
	T(A) =  \sum_{a,b\in A} \E^\times (A-a,A-b) \,,
\end{equation}
and
\begin{equation}\label{f:Q_via_E}
	Q(A) = \sum_{a,b\in A} \E_3^\times (A-a,A-b) \,.
\end{equation}

Next, we state what is known for collinear triples. The contribution to collinear triples coming from the $|A|$ horizontal lines incident to $|A|$ points in $A \times A$ is $|A|^4$. All other collinear triples can be counted by the number of solutions to
\begin{equation}\label{Alg T(A)1}
(b-a) (c'-a') = (b'-a')(c-a) \;, \; a, b , \dots, c' \in A.
\end{equation}
It follows from this that the expected number of (solutions to~\eqref{Alg T(A)1} and therefore of) collinear triples in $A \times A$ where the elements of $A$ are chosen uniformly at random is $\ds \frac{|A|^6}{p} + O(|A|^4)$.

Another interesting example is that of sufficiently small arithmetic progressions: if $A = \{1,\dots, \sqrt{p}/3\} \subset \mathbb{Z}$ then $T(A) \gg |A|^4 \log^\gamma|A|$, for some absolute $0<\gamma<1$, due to Ford \cite{Ford2008}.

Over the reals, Elekes and Ruzsa observed in~\cite{Elekes-Ruzsa2003} that the Szemer\'edi-Trotter point-line incidence theorem~\cite{Szemeredi-Trotter1983} implies
\[
T(A) \ll |A|^4\log|A| .
\]
Because of this and the two examples discussed above, it is possible that the inequality
\[
T(A) \ll \frac{|A|^6}{p} +  |A|^4\log|A|
\]
is correct up to logarithmic factors in $\F_p^2$.

Far less is currently known. It is straightforward to obtain, say from the forthcoming Lemma~\ref{BKT i},
\begin{equation}
\label{TA large A}
\left | T(A) - \frac{|A|^6}{p} \right| = O(p|A|^3). 
\end{equation}
It follows that if $|A| \gg p^{2/3}$, then $T(A) \approx |A|^6 /p.$

Combining this bound with \cite[Proposition~5]{AYMRS} gives
\begin{equation}\label{T(A) known}
T(A) \ll  \frac{|A|^6}{p} + |A|^{9/2},
\end{equation}
however this does not improve the range where $T(A)\ll \frac{|A|^6}p$.

Similar results are true for collinear quadruples. The expected number of collinear quadruples is $\ds \frac{|A|^8}{p^2} + O(|A|^5)$. The above mentioned result of~\cite{AYMRS} implies that $Q(A) \ll |A|^{11/2}$ when $|A| \leq p^{2/3}$. One expects that the correct order of magnitude up to logarithmic factors is
\[
Q(A) \ll  \frac{|A|^8}{p^2} +  |A|^{5}\log|A|. 
\] 
Once again, large random sets and small arithmetic progressions offer (nearly) extremal examples.

We present an improvement on~\eqref{T(A) known} when $|A| \gg p^{1/2}$ and establish a nearly best possible bound for $Q(A)$.

\begin{thm}\label{T(A)}
Let $A \subset \F_{p}$. 
\begin{enumerate}
\item The number of collinear triples in $A \times A$ satisfies 
\[
T(A) = \frac{|A|^6}{p} + O(p^{1/2} |A|^{7/2}).
\] 
So there is at most a constant multiple of the expected number of collinear triples when $|A| \gg
p^{3/5}$.
\item The number of collinear quadruples in $A \times A$ satisfies 
\[ 
Q(A) =  \frac{|A|^8}{p^2} + O( |A|^5 \log|A|),
\] 
which is optimal up to perhaps logarithmic factors.
\end{enumerate}
\end{thm}

The theorem leads to improvements on some of the exponential sum estimates in~\cite{Trinomials,Trilinear}. We do not pursue this direction in the current paper.

The first step in the proof of Theorem~\ref{T(A)} is the proof of inequality~\eqref{TA large A}. 

\begin{proof}[Proof of Inequality~\eqref{TA large A}]
	Using Lemma \ref{BKT i} and the formula $\sum_\ell i(\ell) = (p+1)|A|^2$, one has 
\begin{align*}
\sum_\ell i^3 (\ell)
& = \sum_\ell i(\ell) \left( i(\ell) - \frac{|A|^2}{p} + \frac{|A|^2}{p} \right)^2\nonumber\\
&=\sum_\ell i(\ell) \left( i(\ell) - \frac{|A|^2}{p} \right)^2 + \frac{2|A|^2}{p} \sum_\ell i^2 (\ell) - \frac{|A|^4}{p^2} \sum_\ell i(\ell) \nonumber\\
&= \sum_\ell i(\ell) \left( i(\ell) - \frac{|A|^2}{p} \right)^2 + \frac{|A|^6}{p} + 2 |A|^4 - \frac{|A|^6}{p^2}\,.
\end{align*}
Thus by~\eqref{Ti Qi}
\begin{equation}
\label{TA error}
\left| T(A) - \frac{|A|^6}p -2|A|^4 +  \frac{|A|^6}{p^2} \right| =  \sum_\ell i(\ell) \left( i(\ell)-\frac{|A|^2}p\right)^2 + O(|A|^4).
\end{equation}
Since any line contains at most $|A|$ points of $A\times A$, we have $i(\ell)\leq |A|$ for all $\ell$.
Thus by Lemma~\ref{BKT i},
\[
\sum_\ell i(\ell) \left( i(\ell)-\frac{|A|^2}p\right)^2 \leq |A| \sum_\ell \left( i(\ell)-\frac{|A|^2}p\right)^2 \leq p|A|^3,
\]
which proves Equation~\ref{TA large A}.
\end{proof}
To prove Theorem~\ref{T(A)}, we will improve the last step of the preceding proof by splitting into cases where $i(\ell)\leq \Delta$ and $i(\ell) > \Delta$ for some parameter $\Delta > 0$.
\begin{proof}[Proof of Theorem~\ref{T(A)}]
We prove the following asymptotic estimates for the number of collinear triples and quadruples, from which the theorem follows from $|A| \leq p$:
\begin{equation}
\label{T(A) asymp}
\left| T(A)-\frac{|A|^6}{p} - 2|A|^4 + \frac{|A|^6}{p^2} \right|\ll p^{1/2}|A|^{7/2}
\end{equation}
and
\begin{equation}
\label{Q(A) asymp}
\left| Q(A)-\frac{|A|^8}{p^2}\right|\ll |A|^5\log|A|.
\end{equation}

We begin with~\eqref{T(A) asymp}. By equation~\eqref{TA error}, we have
\[
\left| T(A)-\frac{|A|^6}{p} - 2|A|^4 + \frac{|A|^6}{p^2}\right| = \sum_{\ell}i(\ell)\left( i(\ell)-\frac{|A|^2}p\right)^2 +O(|A|^4).
\]
We will show that
\begin{equation}
\label{TA error 2}
\sum_{\ell}i(\ell)\left( i(\ell)-\frac{|A|^2}p\right)^2 \ll p^{1/2}|A|^{7/2}.
\end{equation}
This will finish the proof of~\eqref{T(A) asymp}, because $|A| \leq p$. 

Let $\Delta>0$ be a parameter, which we will specify later.
Split the sum on the left-hand side of \eqref{TA error 2} into two parts:
\begin{align*}
\sum_{\ell}i(\ell)\left( i(\ell)-\frac{|A|^2}p\right)^2 &= \sum_{i(\ell)\leq\Delta}i(\ell)\left( i(\ell)-\frac{|A|^2}p\right)^2 + \sum_{i(\ell)>\Delta}i(\ell)\left( i(\ell)-\frac{|A|^2}p\right)^2 
\\&= I + II.
\end{align*}
To bound term $I$, we use Lemma~\ref{BKT i}:
\[
I = \sum_{i(\ell)\leq\Delta}i(\ell)\left( i(\ell)-\frac{|A|^2}p\right)^2 \leq \Delta \sum_\ell\left( i(\ell)-\frac{|A|^2}p\right)^2 = \Delta p|A|^2.
\]
To bound term $II$, we first suppose that $\Delta\geq 2|A|^2/p$, an assumption that we will show is satisfied by our choice of $\Delta$.
This allows us to write
\[
II = \sum_{i(\ell)>\Delta}i(\ell)\left( i(\ell)-\frac{|A|^2}p\right)^2 \ll \sum_{i(\ell)>\Delta}i(\ell)^3.
\]
Now we use dyadic decomposition and Lemma~\ref{L_M} to bound the right-hand side of the previous equation:
\begin{align*}
\sum_{i(\ell)>\Delta}i(\ell)^3 &\ll \sum_{k\geq 0} (2^k\Delta)^3|L_{2^k\Delta}|
\ll \sum_{k\geq 0} (2^k\Delta)^3\frac{|A|^5}{(\Delta 2^k)^4}
\ll\frac{|A|^5}{\Delta}.
\end{align*}
(Note that our application of Lemma~\ref{L_M} is justified, since $2^k\Delta\geq 2|A|^2/p$ for all $k\geq 0$.)

All together, we have
\[
\sum_{\ell}i(\ell)\left( i(\ell)-\frac{|A|^2}p\right)^2 \ll \Delta p|A|^2 + \frac{|A|^5}{\Delta}.
\]
Choosing $\Delta=|A|^{3/2}/p^{1/2}$ yields
\[
\sum_{\ell}i(\ell)\left( i(\ell)-\frac{|A|^2}p\right)^2 \ll p^{1/2}|A|^{7/2},
\]
assuming that $|A|^{3/2}/p^{1/2}\geq 2|A|^2/p$.
If not, we have $p < 4|A|$, thus by \eqref{TA large A}, we have
\[
\sum_{\ell}i(\ell)\left( i(\ell)-\frac{|A|^2}p\right)^2 \ll p|A|^3 \ll p^{1/2}|A|^{7/2}.
\]
Thus in any case, we have \eqref{TA error 2}.

Next we prove~\eqref{Q(A) asymp}. Arguing as in the proof of \eqref{TA large A}, by~\eqref{Ti Qi}, Lemma~\ref{BKT i} and the asymptotic formula for $T(A)$, we have
\begin{align*}
Q(A) &= \sum_{\ell}i(\ell)^4 + O(|A|^4)\\
&=\sum_\ell i(\ell)^2 \left(i(\ell)-\frac{|A|^2}p \right)^2 - \frac{2|A|^2}{p}T(A) + \frac{|A|^4}{p^2}(|A|^4+p|A|^2) + O(|A|^4)\\
&=\frac{|A|^8}{p^2}+ O\left( \frac{|A|^6}{p} \right) + O\left( \frac{|A|^{11/2}}{p^{1/2}}\right) + \sum_\ell i(\ell)^2 \left(i(\ell)-\frac{|A|^2}p \right)^2 \,.
\end{align*}
Now we estimate the sum, first by removing the portion where $i(\ell)<2|A|^2/p$ 
\[
 \sum_{i(\ell)<2|A|^2/p} i(\ell)^2 \left(i(\ell)-\frac{|A|^2}p \right)^2\leq \frac{4|A|^4}{p^2}p|A|^2 = \frac{4|A|^6}p,
\]
and then by applying the second part of Lemma~\ref{L_M} to each of the $O(\log|A|)$ remaining dyadic sums of the form $\displaystyle \sum_{ 2 ^i \leq i(\ell) \leq 2^{i+1}} i(\ell)^4$ with $ 2 |A|^2/p \leq 2^i \leq |A|$
\[
\sum_{i(\ell)\geq 2|A|^2/p} i(\ell)^2 \left(i(\ell)-\frac{|A|^2}p \right)^2 \ll \sum_{i(\ell)\geq 2|A|^2/p}i(\ell)^4  \ll |A|^5\log |A| .
\]
Since the terms $|A|^6/p$ and $|A|^{11/2}/p^{1/2}$ are both smaller than $|A|^5$, we have
\[
Q(A) = \frac{|A|^8}{p^2}+O(|A|^5\log |A|).
\]
\end{proof}

\subsection[Proof of part 2 of Theorem 5]{Proof of part 2 of Theorem~\ref{R[A]}}

Recall that part 2 of Theorem \ref{R[A]} states that $|R[A]| \gg \min\left\{ p, \frac{|A|^{5/2}}{p^{1/2}} \right\}.$

\begin{proof}[Proof of Part 2 of Theorem~\ref{R[A]}]
$R[A]$ is the support of the function $t(x)$ defined in~\eqref{t function}. Note also that 
\[
\sum_{x \in \F_p} t(x) = |A|^2(|A|-1)
\]
because each ordered triplet $(a,b,c) \in A \times A \times A$ with $c \neq a$ contributes precisely 1 to the sum.

By the Cauchy-Schwarz inequality we get
\[
|R[A]| = |\supp(t)| \geq \frac{\left( \sum_x t(x) \right)^2}{\sum_x t(x)^2} \gg  \frac{|A|^6}{\sum_x t(x)^2}.
\]
By identity~\eqref{t expression} and Part 1 of Theorem~\ref{T(A)} we have
\[
\sum_{x \in \F_p} t(x)^2 \leq T(A) \ll \frac{|A|^6}{p} + p^{1/2} |A|^{7/2}.
\] 
Substituting above yields
\begin{equation}\label{T(A) R bound}
|R[A]| \gg \min\left\{ p, \frac{|A|^{5/2}}{p^{1/2}} \right\}.
\end{equation}
When $|A| \geq p^{3/5}$, $p \leq \tfrac{|A|^{5/2}}{p^{1/2}}$ and the claim follows. 
\end{proof}

Also note that if $|A| > p^{3/5+\epsilon}$, then $|R[A]| = p(1-\varepsilon)$ for all $\varepsilon >0$.

\subsection[Proof of part 1 and part 3 of Theorem 5]{Proof of part 1 and part 3 of Theorem~\ref{R[A]}}

Recall that Part 1 states that if $|A| \leq p^{5/12}$, then $|R[A]| \gtrsim |A|^{8/5}$; and Part 3 states that $|R[A]| \gtrsim \min\{p, |A|^{\frac{3}{2} + \frac{1}{22} }\}$.  
\begin{proof}[Proof of part 1 and part 3 of Theorem~\ref{R[A]}]
We begin with some preliminaries. We denote $R[A]$ by $R$. Recall also the functions $t$ and $q$ defined in~\eqref{t function} and~\eqref{q function} respectively.
\begin{equation}\label{f:def_t,q}
  t(x)=|\{(a,b,c)\in A^3\colon \tfrac{b-a}{c-a}=x\}| \text{ and } q(x,y)=|\{(a,b,c,d)\in A^4\colon \tfrac{b-a}{c-a}= x , \tfrac{d-a}{c-a} = y\}|.
\end{equation}
It follows from the definitions that
\begin{equation}
  \label{eq:3}
  \sum_{x\in R}\sum_{y\in R} q(x,y) =|A|\sum_{x\in R} t(x) = |A|^3(|A|-1) \leq |A|^4.
\end{equation}

We will exploit a symmetry of the set $R$, namely that $R=1-R$.
To prove this, simply note that
\begin{equation}\label{f:basic_identity}
1-\frac{b-a}{c-a}=\frac{c-b}{c-a}=\frac{b-c}{a-c}.
\end{equation}
This shows that $t(x)=t(1-x)$ and $q(x,y)=q(1-x,1-y)$ as well.

Let $G\subset  R\times R$ denote the set of pairs $(x,y)$ such that $q(x,y)\not=0$.
Applying Cauchy-Schwarz to the inequality in \eqref{eq:3} yields
\begin{equation}
  \label{eq:9}
  |A|^8\ll\left(\sum_{(x, y)\in G}q(x,y) \right)^2\leq |G|\left( \sum_{x,y}q(x,y)^2 \right).
\end{equation}

To bound $|G|$, we use the elementary observation that if $(x,y)\in G$, then there exist $a,b,c,d \in A$ such that $x = \tfrac{b-a}{c-a}$ and $y = \tfrac{d-a}{c-a}$.
It follows that $\tfrac{x}{y} = \tfrac{b-a}{d-a} \in R$.
Since $q(x,y)=q(1-x,1-y)$, we also have $(1-x,1-y)\in G$, hence $\tfrac{1-x}{1-y}\in R$.

Setting $G_x=\{y\colon (x,y)\in G\}$ we have
\[
|G| = \sum_{x\in R}|G_x|\leq |R|^{1/2} \left(\sum_{x\in R}|G_x|^2 \right)^{\tfrac 12}
= |R|^{1/2} |\{(x,y,y')\in R^3\colon (x,y),(x,y')\in G \}|^{1/2}.
\]
If $(x,y),(x,y')\in G$, then there exist elements $r,r'\in R$ such that $x=ry$ and $1-x = r'(1-y')$.
Thus
\[
|\{(x,y,y')\in R^3\colon (x,y),(x,y')\in G \}| \leq |\{(r,r',y,y')\in R\colon 1-ry = r'(1-y')\}|.
\]
The quantity on the right-hand side can be interpreted as the number of incidences between the pairs $(y,y')$ in $R\times R$ and the set of $|R|^2$ lines of the form
\[
1-r' = ry -r'y'.
\]
Applying Theorem~\ref{SdZgen} yields
\[
|\{(r,r',y,y')\in R\colon 1-ry = r'(1-y')\}| \ll \frac{|R|^{7/2}}{p^{1/2}}+|R|^{11/4}.
\]
Thus
\begin{equation}
\label{Gbound}
|G| \ll |R|^{15/8} + \frac{|R|^{9/4}}{p^{1/4}}.
\end{equation}

The term in the second parenthesis of~\eqref{eq:9} is, by~\eqref{q function}, bounded above by the number of collinear quadruples in $A\times A$.
By Part 2 of Theorem~\ref{T(A)} we have
\begin{equation}\label{f:q_calc}
  \sum_{x,y}q(x,y)^2 \ll \frac{|A|^8}{p^2} +|A|^5\log|A|.
\end{equation}
The first term dominates only when $|A|^3 \gg p^2 \log|A|$. This range is out of the scope of Part 1. As for Part 3, Part 2 (proved above) guarantees that $|R| \gg p$ and the conclusion follows. So we may assume that the second term dominates. Therefore,
\begin{equation}
  \label{eq:11}
  \sum_{x,y}q(x,y)^2 \ll |A|^5\log|A|.
\end{equation}

Thus by \eqref{eq:9}, \eqref{Gbound}, and \eqref{eq:11} we have
\begin{equation*}
\frac{|A|^3}{\log|A|} \ll \frac{|R|^{9/4}}{p^{1/4}} + |R|^{15/8}.
\end{equation*}
This implies
\begin{itemize}
\item if $|R| \leq p^{2/3}$, then the second term on the right dominates and 
\begin{equation}
\label{R small}
|R| \gg \frac{|A|^{8/5}}{ \log^{8/15}|A|};
\end{equation}
\item if $|R| \geq p^{2/3}$, then the first term on the right dominates and 
\begin{equation}
\label{R big}
|R| \gg \frac{p^{1/9} |A|^{4/3}}{\log^{4/9}|A|}.
\end{equation}
\end{itemize}

\begin{description}
\item[Conclusion of proof of Part 1:] By \eqref{R small}, either $|R| \gg |A|^{8/5} \log^{-8/15}|A|$ or $|R| \geq p^{2/3}$. The hypothesis $|A| \leq p^{5/12}$ implies that $p^{2/3} \geq |A|^{8/5}$.
\item[Conclusion of the proof of Part 3:]
We combine the lower bounds in \eqref{T(A) R bound} and \eqref{R big}.
They are roughly equal when $|A| = p^{11/21}$.
Using \eqref{T(A) R bound} for $|A| \geq p^{11/21}$ and \eqref{R big} for $|A| \leq p^{11/21}$ gives $
|R|\gg |A|^{17/11}\log^{-4/9}|A|,$ as claimed.\end{description}
\end{proof}

\section[Applications of Theorem 5]{Applications of Theorem~\ref{R[A]}}
\label{sec:RAapp}

This section contains three applications of Theorem~\ref{R[A]}.
The first application is Theorem~\ref{thm:proddiff}, which states that difference sets  are not multiplicatively closed.
The second application is Theorem~\ref{Mult sbgp}, which shows that multiplicative subgroups of $\F_p$ can only contain the difference set of small sets.
The third application is Theorem~\ref{arr}, which is a lower bound for the size of the image of a point set $P\subset \F^2$ under a skew-symmetric form $\omega$.

\subsection{Products of difference sets}
\label{Brendan}
In this subsection we need three additional lemmata. The first lemma is a consequence of adapting the proof \cite[Theorem 3]{GS} by applying Theorem \ref{SdZgen} in place of the Szemer\'{e}di-Trotter Theorem.

\begin{lem} \label{GS} Let $A,B,C \subset \mathbb F$. Suppose, $|AC| \leq |(A-1)B|$ and  $|AC||B||C| \leq p^2$. Then
\[
|AC|^3|(A-1)B|^2 \gg |A|^4|B||C|.
\]
\end{lem}

\begin{proof}
Without loss of generality $0\not\in C$. Let $l_{b,c}$ be the line with equation $y=b(c^{-1}x-1)$ and let $L=\{l_{b,c} : b \in B, c \in C\}$, so that $|L|=|B||C|$. Let $P=AC \times (A-1)B$.
Each line of $L$ is incident to at least $|A|$ points of $P$.

We will apply Theorem \ref{SdZgen} with these sets of points and lines. 
By assumption we have $|AC||B||C| \leq p^2$, and so \eqref{cond2} is satisfied.
Thus
\[
|A||B||C|\leq \cI(P,L) \ll |AC|^{3/4}|(A-1)B|^{1/2}|B|^{3/4}|C|^{3/4} + |AC||(A-1)B|+|B||C|.
\]
If the second term dominates, then the desired lower bound follows from \[|AC|, |(A-1)B|\gg |A|.\]
Thus we may assume that
\[
|A||B||C| \leq \cI(P,L) \ll |AC|^{3/4}|(A-1)B|^{1/2}|B|^{3/4}|C|^{3/4} +|B||C|.
\]
Rearranging this inequality completes the proof.
\end{proof}

The second lemma is the Pl\"{u}nnecke-Ruzsa inequality. A simple proof of the following formulation of the inequality can be found in \cite{P}.
\begin{lem} \label{Plun} Let $A$ be a subset of an abelian group $(G,+)$. Then
\[
|kA-lA| \leq \frac{|A+A|^{k+l}}{|A|^{k+l-1}}.
\]
\end{lem}

We will also use the following adaptation of the Pl\"{u}nnecke-Ruzsa inequality. See \cite[Corollary 1.5]{KaS}.

\begin{lem} \label{CleverPlun} Let $X, B_1,\dots, B_k$ be subsets of an abelian group $(G,+)$. Then there exists $X' \subset X$ such that $|X'| \geq |X|/2$ and
\[
|X'+B_1+\dots B_k| \ll \frac{|X+B_1|\cdots |X+B_k|}{|X|^{k-1}}.
\]
\end{lem}

We are now ready to prove Theorem \ref{thm:proddiff}.
\begin{proof}[Proof of Theorem \ref{thm:proddiff}]
Recall that Theorem~\ref{thm:proddiff} states that for $A\subset \F$ with $|A|^{16}|A-A|^{30}\leq p^{25}$ and $|A|<p^{5/12}$, we have
\[
|(A-A)(A-A)| \gtrsim |A-A|^{4/5} |A|^{\frac{32}{75}}\geq |A-A|^{1+\frac 1{75}}.
\]

The second inequality follows from the trivial inequality $|A| \geq |A-A|^{1/2}$, and so it suffices to prove the first bound only.

We use the shorthand $R=R[A]$ and $D=A-A$. The proof is split into two cases.
\begin{kase} Suppose that 
\begin{equation}\label{08.01.2017_1}
	|RD||D|^2 \leq p^2.
\end{equation}
Apply Lemma \ref{GS} with $A=R$ and $B=C=D$. Note that, since it is assumed that $|RD||D|^2 \leq p^2$, the conditions of Lemma \ref{GS} in the positive characteristic case are satisfied. Then, recalling the property that $R-1=-R$, we have
\[
|RD|^5 \gg |R|^4|D|^2.
\]
We can apply the bound $|R| \gtrsim |A|^{8/5}$ from Theorem \ref{R[A]}. This is valid by our assumption that $|A|< p^{5/12}$. It follows that
\[
|RD|^5 \gtrsim |A|^{\frac{32}{5}}  |D|^{2}.
\]
On the other hand, we can copy arguments from \cite{S_diff} to bound the left side of this inequality from above. Indeed, applying Lemma \ref{Plun} in the multiplicative setting, we have
\[
|RD| \leq |DD/D| \leq \frac{|DD|^3}{|D|^2}.
\]
Therefore,
\[
|DD|^{15} \geq |RD|^5|D|^{10} \gtrsim |A|^{32/5} |D|^{12},
\]
as required.

\end{kase}

\begin{kase}
Suppose that 
\begin{equation}\label{08.01.2017_2}
	|RD||D|^2 \geq p^2.
\end{equation}
Similarly to the previous case, it follows from Lemma \ref{Plun} that
\[
p^2 \leq |RD||D|^2 \leq |DD/D||D|^2 \leq |DD|^3
\]
and so
\[
|DD|^{75} \geq p^{50} \geq |A|^{32}|D|^{60},
\]
where the last inequality follows from the assumption that $|A|^{16}|D|^{30} \leq p^{25}$.
\end{kase}
This completes the proof in the case when $\F$ has characteristic $p$. For fields with characteristic zero the arguments in Case 1 suffice to complete the proof.
\end{proof}

\begin{proof}[Proof of Theorem \ref{thm:superquadratic}]

Recall that Theorem \ref{thm:superquadratic} states that, for any $A \subset \F_p$ with $|A|\leq p^{25/77}$,
\[|DD/D| \gtrsim |A|^{2+\frac{1}{25}},
\]
where again $D=A-A$.
The proof is similar to the proof of Theorem \ref{thm:proddiff}. Write $Q=(A-A)/(A-A)$ and $R=R[A]$. At the outset, apply Lemma \ref{CleverPlun} to obtain a subset $Q' \subset Q$ with 
\begin{equation} \label{CPlun1}
|Q'| \gg |Q| \gg |A|^2
\end{equation}
and
\begin{equation} \label{CPlun2}
\left |\frac{Q'D}{D} \right | \ll \frac{|QD||Q/D|}{|Q|}= \frac{|DD/D|^2}{|Q|}.
\end{equation}
The bound $|Q| \gg |A|^2$ is a consequence of \eqref{szonyi}, a result which is essential for this proof to work. This is valid here because of our assumption that $|A|\leq p^{25/77}<p^{1/2}$.

We will prove, under our assumption on the size of $A$, that
\begin{equation}
|RQ'| \gtrsim |A|^{2+\frac{2}{25}}.
\label{intermediate}
\end{equation}
Suppose that \eqref{intermediate} holds. Then in particular, by \eqref{CPlun2} we have
\[ |A|^{2+\frac{2}{25}} \lesssim |RQ'|\leq \left| \frac{Q'D}{D} \right| \ll \frac{|DD/D|^2}{|Q|}.
\]
Rearranging this inequality and applying the bound $|Q| \gg |A|^2$, we obtain the desired result that
\[|DD/D| \gtrsim |A|^{2+\frac{1}{25}}.
\]

It remains to prove \eqref{intermediate}. We may assume that $|RQ'| \leq |A|^{2+\frac{2}{25}}$, as otherwise we are done. Also, we may assume that $|Q| \leq |A|^{2+\frac{1}{25}}$, as otherwise there is nothing to prove. With these two assumptions and the assumption that $|A|\leq p^{25/77}$, we have
\[|RQ'||Q'|^2 \leq |A|^{6+\frac{4}{25}} \leq p^2,
\]
and so Lemma \ref{GS} can be applied to obtain the bound
\[
|RQ'|^5=|RQ'|^3|(R-1)Q'|^2 \gg |R|^4|Q'|^2.
\]
We can apply the bound $|R| \gtrsim |A|^{8/5}$ from Theorem \ref{R[A]}. This is valid by our assumption that $|A|< p^{5/12}$. In addition, we can apply the bound on $|Q'|$ from \eqref{CPlun1}. Putting everything together, it follows that
\[
|RQ'|^5 \gtrsim |A|^{52/5}.
\]
This proves \eqref{intermediate}, and thus completes the proof of Theorem \ref{thm:superquadratic}.
\end{proof}

\subsection{Application to multiplicative subgroups} \label{Ilya}

In this section we show that Theorem \ref{R[A]} implies a new bound for the size of a set $A$ such that $A-A$ is contained in the union of a multiplicative subgroup and $\{0\}$.
Our Theorem \ref{t:A-A_in_G} below improves Theorem 25 from \cite{S_diff}. In~\cite{Shkredov2017} Theorem~\ref{T(A)} is used to prove that small multiplicative subgroups are additive irreducible.

We need a result on intersections of additive shifts of subgroups, e.g., see \cite{sv}.

\begin{thm}
    Let $\Gamma\subset \F_p$ be a multiplicative subgroup,
    $k\ge 1$ be a positive integer, and $x_1,\dots,x_k$ be different nonzero elements.
    Also, let
    \[
        32 k 2^{20k \log (k+1)} \le |\Gamma|\,, \quad  p \ge 4k |\Gamma|  ( |\Gamma|^{\frac{1}{2k+1}} + 1 ) \,.
    \]
    Then
    \begin{equation}\label{f:main_many_shifts}
        |\Gamma\cap (\Gamma+x_1) \cap \dots (\Gamma+x_k)| 
        	\le 4 (k+1) (|\Gamma|^{\frac{1}{2k+1}} + 1)^{k+1} \,.
    \end{equation}
    Further 
    \begin{equation}\label{f:C_for_subgroups}
    	|\Gamma \cap (\Gamma+x_1) \cap \dots \cap (\Gamma+x_k)| 
    		= \frac{|\Gamma|^{k+1}}{(p-1)^k} + \theta k 2^{k+3} \sqrt{p} \,,
	\end{equation}
    where $|\theta| \le 1$.
    The same holds if one replaces  $\Gamma$ in (\ref{f:main_many_shifts}) by any cosets of $\Gamma$.
\label{t:main_many_shifts}
\end{thm}

Now let us formulate the main result of  this section.

\begin{thm}\label{Mult sbgp}
    Let $\Gamma \subset \F_p$ be a multiplicative subgroup and $\xi \neq 0$ be an arbitrary residue.
Suppose that for some $A\subset \F_p$ one has
\begin{equation}\label{cond:A-A_new}
    A-A \subset \xi \Gamma \sqcup \{ 0 \} \,.
\end{equation}
    Then 
    \begin{itemize}
    \item If $|\Gamma| < p^{3/4}$ then $|A| \lesssim |\Gamma|^{5/12} $.
    \item If $p^{3/4} \leq |\Gamma| < p^{5/6}$ then $|A| \lesssim p^{-5/8}|\Gamma|^{5/4}$.
\item If $|\Gamma| \ge p^{5/6}$, then 
	$
    	|A| \lesssim p^{-1} |\Gamma|^{5/3} 
    $.

    \end{itemize}

    In particular, for any $\varepsilon>0$ and sufficiently large $\Gamma$, $|\Gamma| \le p^{6/7-\varepsilon}$ the following holds
    \[
        A-A \neq \xi \Gamma \sqcup \{ 0 \} \,.
    \]
\label{t:A-A_in_G}
\end{thm} 
\begin{proof}
    We can assume that $|A|>1$.
Put $\Gamma_* = \Gamma \sqcup \{ 0 \}$ and $R=R[A]$.
Then in light of our condition (\ref{cond:A-A_new}), we have
\begin{equation}\label{f:R_inclusion}
    R\subset (\xi \Gamma \sqcup \{ 0 \} ) / \xi \Gamma = \Gamma_* \,.
\end{equation}
Suppose that $|\Gamma| \le p^{3/4}$. 
Applying (\ref{f:basic_identity}) and formula (\ref{f:main_many_shifts}) with $k=1$, we obtain
\[
    |R| = |R \cap(1-R)| \le |\Gamma_* \cap (1-\Gamma_*)| \le  |\Gamma \cap (1-\Gamma)| + 2 	\ll |\Gamma|^{2/3} \le p^{1/2} \,.
\]
Since $|A|\leq |R| < p^{1/2}<p^{5/12}$, we may apply formula (\ref{R small}) of Theorem \ref{R[A]} to get $|R| \gtrsim |A|^{8/5}$.
Hence $|A|^{8/5} \lesssim |\Gamma|^{2/3}$,
and we obtain the required result.

Now let us suppose that $|\Gamma| \ge p^{3/4}$ but $|\Gamma| \le p^{5/6}$.
In this case, by formula (\ref{f:C_for_subgroups}) and the previous calculations, we get
\[
	|R| \ll |\Gamma|^2 / p \le p^{2/3} \,.
\]
Applying Theorem \ref{R[A]} (namely, formula (\ref{R small})) one more time, we obtain 
$|A|^{8/5} \lesssim |\Gamma|^2 / p$, 
which gives the required bound for the size of $A$, namely $|A| \lesssim p^{-5/8} |\Gamma|^{5/4}$.

Now suppose that $|\Gamma|\ge p^{5/6}$.
We use the second formula from (\ref{f:def_t,q}) and note that if $q(x,y)>0$, then $x-y \in \Gamma_*$.
By (\ref{eq:9})--(\ref{f:q_calc}) and inclusion (\ref{f:R_inclusion}), we obtain
\[ \begin{aligned}
	|A|^8 
    	 & \ll 
        	\left( \frac{|A|^8}{p^2} + |A|^5 \log |A| \right)
            	\cdot
    		\left( \sum_{x,y} R (x) R(y) \Gamma_* (x-y) \right) \\
            	& \lesssim
    	\left( \frac{|A|^8}{p^2} + |A|^5  \right)
            	\cdot
	\left( \sum_{x} \Gamma_* (x) \Gamma_* (1-x) \sum_y \Gamma_*(y) \Gamma_* (1-y) \Gamma_* (x-y) \right) \,. \end{aligned}
\]
It is easy to see that the summand with $x=1$ is negligible in the last inequality.  
Using formula (\ref{f:C_for_subgroups}) with $k=2$ and the condition $|\Gamma|\ge p^{5/6}$, we get
\[
	|A|^8 
	\lesssim
        	\left( \frac{|A|^8}{p^2} + |A|^5 \right) \cdot \frac{|\Gamma|^5}{p^3} 
            	\ll
                	|A|^5  \cdot \frac{|\Gamma|^5}{p^3} \,.
\]
It follows that $|A| \lesssim p^{-1} |\Gamma|^{5/3}$, as claimed.

Finally, it is known that if $A+B = \xi \Gamma$ or $A+B = \xi \Gamma \sqcup \{ 0 \}$, then $|A| \sim |B| \sim |\Gamma|^{1/2+\varepsilon}$ for all $\varepsilon >0$, see \cite{Shparlinski_AD}, \cite{sv}. 
Thus, if $p^{5/6} \le |\Gamma| \le p^{6/7-\varepsilon}$, 
we have $|A| \le |\Gamma|^{1/2-\varepsilon/2}$ which  is a  contradiction.  
Another two bounds which were obtained before give us the same for smaller $\Gamma$. Thus, we have $\xi \Gamma \sqcup \{ 0 \} \neq A-A$ for sufficiently large $\Gamma$.

\end{proof}

\subsection[Proof of Theorem 4]{Proof of  a weaker version of Theorem \ref{arr}} \label{Misha}

Recall that we assume $\omega$ is defined by
\[
\omega[(u_1,u_2),(v_1,v_2)]=u_1v_2-u_2v_1,
\]
and that $\omega(P)=\{\omega(p,p')\colon p,p'\in P \}\setminus{0}$.

In this section, we will show that  {\it if $|P|\leq p^{191/192}$ and $P$ is not supported on a line through the origin, then }
\begin{equation}
\label{arrCheap}
|\omega(P)|\gg |P|^{128/191},
\end{equation}
which is slightly weaker than Theorem~\ref{arr}. 

We have chosen to do so to fix the ideas. The full proof of Theorem~\ref{arr} is somewhat more technical,  so we have relegated it to Appendix C.

We split the intermediate claims into short lemmata.
\begin{lem} \label{both} Consider a point set $P\subset \F^2\setminus\{0\}$ that is not contained in a single line.
If $P$ contains at most $w$ points per line, where $w\geq 1$, and $|P| < p\sqrt{w}$, then $|\omega(P)|\gg |P|/\sqrt{w}$.
\end{lem}

\begin{lem}
\label{omega_upper_bnds}
Let $T \subset \F$, with $|T|\leq p^{2/3}$. Then the number of solutions of the equation
\[
 \omega(u,u')=\omega(v,v'),\qquad u,u',v,v' \in T\times T
 \]
is $O(|T|^{13/2})$ and the number of solutions to
\[
\omega(u,u')=s
\]
for $s\not=0$ fixed is $O(|T|^{11/4})$.
\end{lem}
Let \[r_{TT-TT}(s)=|\{(t_1,t_2,t_3,t_4)\in T^4\colon t_1t_2-t_3t_4=s\}|.\]
Then Lemma~\ref{omega_upper_bnds} states that
\begin{equation}
\label{omega_second_moment}
\sum_s r^2_{TT-TT}(s) \ll |T|^{13/2}
\end{equation}
and for $s\not=0$
\begin{equation}
\label{omega_pointwise}
r_{TT-TT}(s)\ll |T|^{11/4}.
\end{equation}
For $A\subset \F \mathbb{P}^1$, define
\[
C[A] = \left\{ \frac{(a-b)(c-d)}{(a-c)(b-d)}:\, \mbox{pairwise distinct } \;a,b,c,d\in A \right\}\subset \F^*.
\]
to be the set of cross-ratios, generated by quadruples of elements of $A$.
\begin{lem} \label{crt}
For any $A\subset \F\mathbb{P}^1$ with $4\leq |A|\leq p^{5/12}$ one has 
\[
|C[A]|\gg \frac{|A|^{8/5}}{\log^{8/15} |A|}.
\]
\end{lem}

The first lemma and the first bound in the second lemma are implicit in the proof of Theorem 13 in \cite[Section 6.1]{Rudnev}.
To prove \eqref{omega_pointwise}, apply Theorem~\ref{SdZgen} with $P=T\times T$ and lines of the form $t_1x+t_2y=s$ with $t_1,t_2$ in $T$.

The third lemma is an immediate corollary of Claim 1 in  Theorem \ref{R[A]}; alternatively, note that for any $a\in A$ one has $R[(A-a)^{-1}] \subset C[A]$, see Note \ref{note:T_C} below.

\begin{proof}[Proof of \eqref{arrCheap}]

Without loss of generality, we may assume that, say at most $|P|^{3/4}$ points of $P\subset \F^2\setminus\{0\}$ lie on a line, since otherwise a fixed point $u\in P$ outside this line determines $\geq |P|^{3/4}$ distinct values of $\omega(u,u')$ with $u'$ on the line, as $\omega$ is non-degenerate. This is needed just to ensure in the sequel that $P$ or its large subset determines at least four distinct directions.

First, we use dyadic pigeonholing to find a subset of $P_1\subset P$ that is uniformly distributed on lines through the origin.
Let $D$ be the set of slopes of the lines through the origin incident to $P_1$ and let $\mu>0$ be such that any line $\ell$ whose slope is in $D$ intersects $P_1$ in roughly $\mu$ points:
\[
\mu\leq |\ell\cap P_1| < 2\mu.
\]
We have $|P_1|\gg |P|/\log |P|\gtrsim |P|$ since $\mu\leq |P|$. 

By Lemma~\ref{both},
\begin{equation}
\label{alt_one}
|\omega(P_1)| \gg \frac{|P_1|}{\sqrt{\mu}}\gtrsim \frac{|P|}{\sqrt{\mu}}.
\end{equation}

We will balance \eqref{alt_one} with a second bound that takes advantage of an algebraic relation between certain elements of $\omega(P)$, reflecting the well-known symmetry of the cross-ratio. The next two lemmata prepare the ground.

If $a,b\in \F^2\setminus\{0\}$, let $t_{ab}=\omega(a,b)$.
For two quadruples $q = (a,b,c,d)$ and $q'=(a',b',c',d')$, where the points $a,b,c,d$ (as well as $a',b',c',d'$) lie in four distinct directions through the origin,  we say that $q\sim q'$ if and only if
\[
(t_{ac},t_{bd},t_{ab},t_{cd}) = (t_{a'c'},t_{b'd'},t_{a'b'},t_{c'd'})
\]
For a point $a$ in $\F^2\setminus\{0\}$, let $\delta_a\in \F \mathbb{P}^1$ denote the slope of the line through the origin and $a$. 
\begin{lem}
\label{qequiv}
Two quadruples $(a,b,c,d)$ and $(a',b',c',d')$ are equivalent if and only if the cross-ratios $[\delta_a,\delta_b,\delta_c,\delta_d]$ and $[\delta_{a'},\delta_{b'},\delta_{c'},\delta_{d'}]$ are equal:
\[
(a,b,c,d)\sim(a',b',c',d') \iff [\delta_a,\delta_b,\delta_c,\delta_d]=[\delta_{a'},\delta_{b'},\delta_{c'},\delta_{d'}].
\]
\end{lem}

\begin{proof}[Proof of Lemma~\ref{qequiv}]
Note that, 
\begin{equation} \label{oif} (a,b,c,d)\sim(a',b',c',d')\;\; \mbox{ only if } \;\; \frac{t_{ab}t_{cd}}{t_{ac}t_{bd}} = \frac{t_{a'b'}t_{c'd'}}{t_{a'c'}t_{b'd'}}. \end{equation}
The quantities in the right side are cross-ratios. To see why we express points in coordinates like $a = (a_1,a_2)$ and by choosing coordinate axes assume that $a_1,b_1,c_1$ and $d_1$ are all non-zero. We then have
\[
\frac{t_{ab}t_{cd}}{t_{ac}t_{bd}} = \frac{(a_1b_2 - a_2 b_1)(c_1 d_2 - c_2 d_1)}{(a_1 c_2 - a_2 c_1)(b_1 d_2 - b_2 d_1)}.
\]
Now, denote by, say, $\delta_a = a_2/a_1$ the direction of the line through the origin that passes through $a$. Dividing both numerator and denominator by $a_1 b_1 c_1 d_1$ gives
\[
\frac{t_{ab}t_{cd}}{t_{ac}t_{bd}} = \frac{(a_1b_2 - a_2 b_1)(c_1 d_2 - c_2 d_1)}{(a_1 c_2 - a_2 c_1)(b_1 d_2 - b_2 d_1)} = \frac{(\delta_b - \delta_a)(\delta_d - \delta_c)}{(\delta_c - \delta_a)(\delta_d - \delta_b)} = [\delta_a, \delta_b, \delta_c, \delta_d].
\]

That is, if $\delta_a,\ldots,\delta_{d'}\in \F \mathbb{P}^1$ are  the directions from the origin, corresponding to the points $a,\ldots,d'$, respectively, then $(a,b,c,d)\sim(a',b',c',d')$ only if $[\delta_a,\delta_b,\delta_c,\delta_d] = [\delta_{a'},\delta_{b'},\delta_{c'},\delta_{d'}],$ where $[\,]$ stands for the cross-ratio.\footnote{Over the reals this also follows from expressing the triangle area as the product of side lengths times the sine of the angle; in general formulation it is often taken for the definition of the cross-ratio of four concurrent lines, see e.g. \cite[Chapter 1]{SK}.}
\end{proof}

For a quadruple of points $a,b,c,d$ in $\F^2\setminus\{0\}$, all lying in distinct directions through the origin, define the map $\Phi$ by
\begin{equation}\label{pmap}
\Phi(a,b,c,d) = (t_{ad},t_{bc},t_{ac},t_{bd},t_{ab},t_{cd}).
\end{equation}

\begin{lem}
\label{phi}
The map $\Phi$ is two-to-one injective on sets of quadruples with the same direction quadruple $(\delta_a,\delta_b,\delta_c,\delta_d)$, and the image of $\Phi$ in $\F^6$ satisfies the equation
\begin{equation}
\label{t_eq}
t_1t_2 = t_3t_4 - t_5t_6.
\end{equation}
Two quadruples $(a,b,c,d)$ and $(a',b',c',d')$ with the same direction quadruple are equivalent if and only if $(a',d')=x(a,d)$ and $(b',c')=x^{-1}(b,c)$ for some $x\not=0$.

\end{lem}

\begin{proof}[Proof of Lemma~\ref{phi}]
The equation \eqref{t_eq} is the identity
\[
t_{ad}t_{bc}=t_{ac}t_{bd}-t_{ab}t_{cd}.
\]
There are many ways to verify this identity: one can do it, e.g.  by direct calculation using coordinates (both sides equal $a_1b_2 c_1 d_2 + a_2 b_1 c_2 d_1 - a_1 b_1 c_2 d_2 - a_2 b_2 c_1 d_1$), trigonometry, or  using the cross and scalar product notation, using the fact that
 \[
(a\times d)\cdot (b\times c) = -c\cdot(b\times(a\times d)) = (a\cdot b)(c\cdot d) - (a\cdot c)(b\cdot d). 
 \]
 Now replace $b=(b_1,b_2)$ by $b^\perp=(-b_2,b_1)$, and $c$ by $c^\perp$.
 
Now we will show that $\Phi$ is two-to-one injective on equivalent quadruples with the same direction quadruple $(\delta_a,\delta_b,\delta_c,\delta_d)$.
Suppose that $(a,b,c,d)$ and $(a',b',c',d')$ are two such quadruples.
Set $a'= k a,$ $b'= l b$, $c'=mc,$ $d'=nd$ with some $k,l,m,n\neq 0$. Then 
$t_{ab}=t_{a'b'}$ means $kl=1$; $t_{cd}=t_{c'd'}$  means $mn=1$. In addition, $t_{ac} = t_{a'c'}$ means $km = 1$ and $t_{bd} = t_{b'd'}$ means $ln=1$.
So $k,n=x$ and $l,m=x^{-1}$, for some $x$.

Informally, given the direction quadruple $(\delta_a,\delta_b,\delta_c,\delta_d)$, an equivalence class of the quadruple of points $(a,b,c,d)$ supported in the corresponding directions arises by  scaling $a,d$ by some factor $x$ and $b,c$ simultaneously by the factor $x^{-1}$.  

So we have $(a',d')=x(a,d)$ and $(b',c')=x^{-1}(b,c)$ for some $x\not=0$.
Thus
\[
t_{a'd'}=x^2t_{ad}\qquad\mbox{and}\qquad t_{b'c'}=x^{-2}t_{bc},
\]
so if $t_{a'd'}=t_{ad}$, then $x=\pm 1$ and $(a,b,c,d)=(a',b',c',d')$.
\end{proof}

We are now ready to balance \eqref{alt_one}. In the notation defined above, $D=\{\delta_a\colon a\in P_1\}$, and we can clearly assume $|D|\geq 4$. Let us also assume for now that $|D|\leq p^{5/12}$.

Then by Lemma~\ref{crt}, quadruples $(a,b,c,d)$ of  points of $P_1$ determine 
\begin{equation}
\label{crt2}
\gtrsim |D|^{8/5} \gg \frac{|P|^{8/5}}{\mu^{8/5}}
\end{equation}
distinct cross-ratios $[\delta_a,\delta_b,\delta_c,\delta_d]$ and hence the same number of distinct equivalence classes of quadruples.

For each equivalence class, fix a direction quadruple $(\delta_a,\delta_b,\delta_c,\delta_d)$.
Lemma~\ref{phi} implies that we have $\gg \mu^4$ solutions to \eqref{t_eq} for each class.
All together, we have
\begin{equation}
\label{t_eq_lower}
\gtrsim |P|^{8/5}\mu^{12/5}
\end{equation}
solutions to \eqref{t_eq} with $t_1,\ldots,t_6$ in $T=\omega(P_1)\setminus\{0\}$.

Writing
\[
r_{TT-TT}(s)=|\{(t_3,\ldots,t_6)\in T^4\colon t_3t_4-t_5t_6\}|,
\]
we have
\[
\sum_{s\not=0}r_{TT}(s)r_{TT-TT}(s) \gtrsim |P|^{8/5}\mu^{12/5}
\]
By \eqref{omega_pointwise}, if $|T|\leq p^{2/3}$ then
\[
r_{TT-TT}(s)\ll |T|^{11/4}
\]
for $s\not=0$.
Thus
\[
\sum_{s\not=0}r_{TT}(s)r_{TT-TT}(s) \ll |T|^{11/4}\sum_{s\not=0}r_{TT}(s) = |T|^{19/4}.
\]
Combining this with \eqref{t_eq_lower} yields
\[
|T|\gtrsim |P|^{32/95}\mu^{48/95}.
\]

Now we prove \eqref{arrCheap} by balancing this lower bound with \eqref{alt_one}:
\[
|\omega(P)| \gtrsim \max\left( |P|\mu^{-1/2},|P|^{32/95}\mu^{48/95}\right) \gg |P|^{128/191}.
\]

To finish the proof, we need to check the conditions required to apply the lemmata.
To apply Lemma~\ref{crt} to $D$, we needed $|D|\leq p^{5/12}$.
Since $|D|\approx |P|/\mu$, if this upper bound does not hold, then
\[
\frac{1}{\mu^{1/2}} \gg \frac{p^{5/24}}{|P|^{1/2}}.
\]
Combined with \eqref{alt_one}, this implies
\[
|\omega(P)| \gg |P|^{1/2}p^{5/24} \geq |P|^{17/24} \gg |P|^{128/191}.
\]
In addition, to apply \eqref{omega_pointwise} we need $|\omega(P_1)|\leq p^{2/3}$; if this fails then \eqref{arrCheap} follows as long as $p^{2/3}\geq |P|^{128/191}$; that is, $|P| \leq p^{191/192}$.
\end{proof}

\section{Sum-product results} \label{Olly}
\label{sec:SPresults}
Starting with Elekes'  proof of the sum-product estimate
\begin{equation}
\label{elekes54}
|A+A|+|AA|\gg |A|^{5/4},
\end{equation}
the best sum-product results over fields have followed from geometric methods.
Elekes' proof of \eqref{elekes54} was based on the Szemer\'edi-Trotter incidence bound \cite{Szemeredi-Trotter1983,To}.
\begin{thm}
For a finite set $A\subset \R$ let $L_k$ denote the set of lines in $\R^2$ that contain at least $k$ points of $P=A\times A$.
Then
\begin{equation}
\label{krich}
|L_k|\ll\frac{|A|^4}{k^3}.
\end{equation}
\end{thm}
The bound \eqref{krich} implies that \(T(A)\ll |A|^4\log|A|\),
which Elekes and Ruzsa used to prove the energy bound
\begin{equation}
\label{ERbound}
\E^\times(A) \ll \frac{|A+A|^4}{|A|^2}\log|A|.
\end{equation}
When $|A+A|\ll |A|$, the bound \eqref{ERbound} implies a sharp bound for the for the product set: \(|AA| \gg |A|^2/\log |A|\) 

The bound \eqref{ERbound} is unusual in that a single application of the Szemer\'edi-Trotter bound yielded a sharp result.
This is not always the case.
For instance, if $A\subset \R$ is a \emph{convex set}\footnote{A set $A$ is called \emph{convex} if there exists a strictly convex function $f\colon\R\to \R$ such that $A=\{f(1),f(2),\ldots,f(n)\}$} then \eqref{krich} yields the energy bound
\(\E^+(A) \ll |A|^{5/2}\),
which implies that
\begin{equation}
\label{convexENR}
|A-A|\gg |A|^{3/2}.
\end{equation}
On the other hand, Schoen and the fifth listed author~\cite{SS} observed that \eqref{krich} implies a sharp (up to logarithms) bound on the third moment of the representation function $r_{A-A}$,  namely
\begin{equation}
\label{convex_third}
\E^+_3(A) \ll |A|^3\log |A|.
\end{equation}
Combined with H\"older's inequality \eqref{convex_third} only recovers \eqref{convexENR} with a logarithmic loss. However \cite{SS} combining \eqref{convex_third} with a second application of the Szemer\'edi-Trotter bound yields
\begin{equation}
\label{convexSS}
|A-A|\gtrsim |A|^{8/5}.
\end{equation}

We use similar arguments to prove Theorem~\ref{thm:fsmp}.
Theorem~\ref{SdZgen} implies that for $A\subset \F_p$ sufficiently small
\begin{equation}
\label{SdZkrich}
|L_k|\ll\frac{|A|^5}{k^4},
\end{equation}
which yields a sharp bound on collinear quadruples $Q(A)\ll |A|^5\log |A|$.
Using quadruples in place of triples, Elekes and Ruzsa's method shows that
\begin{equation}
\label{fsmp_third_bound}
\E_3^\times(A) \ll \frac{|A+A|^5}{|A|^3}\log |A|,
\end{equation}
which is sharp when $|A+A|\ll |A|$.

Combining \eqref{fsmp_third_bound} with H\"older's inequality yields $|AA|\gg |A|^{3/2}/\log|A|$ when $|A+A|\ll |A|$, which can be proved by simpler means.
To take full advantage of the sharp third energy bound \eqref{fsmp_third_bound}, we use the techniques of \cite{SS,SS_moments} to input additional sum-product arguments, yielding Theorem~\ref{thm:fsmp}:
\[
|A/A|\gtrsim |A|^{8/5}  ,\,\,\,\,\,\, |AA| \gtrsim |A|^{14/9}
\]
when $|A+A|\ll |A|$. We describe this argument in more detail in the next section.

Theorem~\ref{thm:fsmp} and \eqref{convexSS} are examples of going beyond the ``$3/2$ threshold" --- improving lower bounds of the form $\gg |A|^{3/2}$ that often follow from straight-forward applications of incidence bounds.
Other examples of results that break the $3/2$ threshold by more sophisticated arguments include the fifth listed author's lower bound \cite{S_diff}
\[
|A-A|\gtrsim |A|^{5/3}
\]
for $A\subset \R$ satisfying $|A/A|\ll |A|$, and 
the fifth listed author's and Vyugin's lower bound \cite{sv} for small multiplicative subgroups $\Gamma$ of $\F_p$:
\[
|\Gamma-\Gamma|\gtrsim |\Gamma|^{5/3},
\]
which improves the lower bound of $|\Gamma-\Gamma|\gg |A|^{3/2}$ established by Heath-Brown and Konyagin \cite{HbK}.

\subsection[Proof of Theorem 2]{Proof of Theorem \ref{thm:fsmp}}

The sum-product estimates proven in \cite{RRS} and \cite{AYMRS} imply the following two statements for all $A \subset \mathbb F$.
\begin{align*}
|A+A|  = M|A|& \;\;\Rightarrow \;\;|AA| \gg \frac{|A|^{\frac{3}{2}}}{M^{\frac{3}{2}}},
\\ |AA| = M|A|& \;\;\Rightarrow\;\; |A+A| \gg \frac{|A|^{\frac{3}{2}}}{M^{\frac{3}{2}}} \,,
\end{align*}
provided that $M|A|^3 \leq p^2$.
In this section, we will improve the first of the statements, in the form of  Theorem \ref{thm:fsmp}.

The method of proof of Theorem \ref{thm:fsmp} is as follows. Firstly, we modify an argument of Solymosi \cite{So}, by using Theorem \ref{SdZgen} in place of the Szemer\'{e}di-Trotter Theorem. This leads to an upper bound on the the third moment multiplicative energy. Recall that this is defined as
\[
\E_3^\times(A)=\sum_{x \in A/A} r_{A/A}^3(x).
\]
\begin{lem} \label{soly2} Let $A \subset \mathbb F$ and $\lambda \in \F^*$. Suppose $|A|^{11}|A+ \lambda A|\leq p^8$. Then
\[
\E_3^\times(A) \ll \frac{|A+ \lambda A|^{15/4}}{|A|^{3/4}} \log|A|.
\]
\end{lem}

In particular, the upper bound has the advantage that it is essentially optimal in the case when $|A \pm A|$ is very small. Let us see how Lemma \ref{soly2} can be used to prove Theorem \ref{thm:fsmp}, using the methods of \cite{SS} and \cite{SS_moments}. The proof of Lemma \ref{soly2} is given in Appendix \ref{sec:soly2proof}.

Recall that, if $A$ is a finite subset of an abelian group, written additively, we define
\[
\E_3(A)=\sum_{d \in A-A}r_{A-A}^3(d) = \left| \{(a_1,\ldots,b_3)\in A\times\ldots \times A:\, a_1-b_1 = a_2-b_2 = a_3-b_3\} \right|.
\]
For $d \in A - A$, we also use the notation
\[
A_d = A \cap (A+d)
\]
and note that $|A_d| = r_{A-A}(d)$.

The following statement is largely part of   \cite[Corollary 3]{SS}, see also \cite{SS_moments}. One can remove some logarithms in (\ref{sums}). 

\begin{lem}\label{shsh} Let $A$ be a subset of an abelian group.
Let $D'\subset A-A$. Then
\begin{equation}
\label{th} \E(A,A\pm A)\geq \sum_{d\in D'} |A_d||A \pm A_d| \geq \frac{|A|^2\left( \sum_{d\in D'}  |A_d|^{\frac{3}{2}}\right)^2 }{\E_3(A)}.
\end{equation}
In particular,
\begin{eqnarray}  \E(A,A-A) & \geq & \frac{|A|^8}{|A-A|\E_3(A)}, \label{difs} \\ \hfill \nonumber \\
\E(A,A+A) & \gg & \frac{|A|^{14}}{|A+A|^3 \E^2_3(A)} . \label{sums}\end{eqnarray}
\end{lem}
\begin{proof} For the proof of the first two relations  see  Lemma 4.3 and formula the proof of Corollary 4.4 in \cite{KR}. Inequality \eqref{sums} follows by applying  the H\"older  and  then Cauchy-Schwarz inequality. We set $D'=A-A$ and $\E_{3/2}(A)=\sum_{d\in A-A}  |A_d|^{3/2}.$ Then
\[
\E(A)  =   \sum_{d\in A-A} r_{A-A}(d) r_{A-A}(d) \leq \E_3^{1/3} \E_{3/2}^{2/3}(A).
\]
Thus $E_{3/2}(A)^2 \geq \E^3(A)/\E_3(A)$, substituting this into \eqref{th} and eliminating $\E(A)$ via
\[
\E(A)\geq \frac{|A|^4}{|A+A|} 
\] yields \eqref{sums}. \end{proof}

Note that in the proof of (\ref{sums}) we have established the energy inequality 
\[
\E(A,A+A) \gg \frac{|A|^2 \E^3 (A)}{\E^2_3 (A) }.
\]
We also need the following lemma, which is implicit in \cite[proof of Proposition 1]{AYMRS}.

\begin{lem} \label{energy} Let $A,B \subset \mathbb F$ such that $|A||A + \lambda A||B| \leq p^2$ and $\max(|A + \lambda A|,|B|)\leq (|A||A + \lambda A||B|)^{1/2}.$ Then
\[ 
\E^\times(A,B) \ll \frac{|A + \lambda A|^{3/2}|B|^{3/2}}{|A|^{1/2}}.
\]
\end{lem}

We can now prove Theorem \ref{thm:fsmp}, assuming that Lemma \ref{soly2} holds.

\begin{proof}[Proof of Theorem \ref{thm:fsmp}]

The first task is prove inequality \eqref{fsmp1}. Apply respectively \eqref{difs} in the multiplicative setting, Lemma \ref{soly2} and Lemma \ref{energy} as follows:
\begin{align} \label{almost}
\begin{split}
|A|^8 &\leq \E_3^\times(A)|A /A|\E^\times(A,A /A) 
\\ &\ll \frac{|A + \lambda A|^{15/4}|A /A|\E^\times(A,A /A)}{|A|^{3/4}}\log |A| 
\\ &\lesssim \frac{|A + \lambda A|^{21/4}|A /A|^{5/2}}{|A|^{5/4}}.
\end{split}
\end{align}

When applying Lemmata \ref{soly2} and \ref{energy}, we must ensure that the conditions of these lemmata are satisfied in our situation. For Lemma \ref{soly2}, this is immediate, since $|A|^{11}|A + \lambda A| \leq |A|^{13} \leq (p^{5/9})^{13} < p^8.$

Now we verify the conditions of Lemma \ref{energy}. We may assume that $|A + \lambda A|^{21}|A / A|^{10} \leq |A|^{37}$, or else there is nothing to prove. Similarly, we may assume that $|A / A| \leq |A|^{8/5}$. Therefore, $|A||A + \lambda A||A / A| = |A| (|A + \lambda A|^{21}|A / A|^{10})^{1/21}|A /A|^{11/21}\leq |A|^{378/105} \leq p^2$, where the last inequality comes from the assumption that $|A| \leq p^{5/9}$. The other condition in Lemma \ref{energy} is trivial in this case.

Rearranging inequality \eqref{almost}, it follows that
\[|A + \lambda A|^{21}|A /A|^{10} \gtrsim |A|^{37}.
\]

For the proof of \eqref{fsmp2}, we apply respectively \eqref{sums} in the multiplicative setting, Lemma \ref{soly2} and Lemma \ref{energy} to obtain
\begin{align*}
|A|^{14} &\leq (\E_3^\times(A))^2|AA|^3 \E^\times(A,AA)
\\ &\lesssim \frac{|A + \lambda A|^{15/2}|AA|^3 \E^\times(A,AA)}{|A|^{3/2}} 
\\ &\leq \frac{|A + \lambda A|^{9}|AA|^{9/2}}{|A|^{2}}.
\end{align*}
Again, it is necessary to check the conditions of Lemma \ref{soly2} and \eqref{energy}. For Lemma \ref{soly2}, this is immediate, since $|A|^{11}|A + \lambda A| \leq |A|^{13} \leq (p^{9/16})^{13} < p^8.$ For Lemma \ref{energy}, we may assume that $|A + \lambda A|^{18}|AA|^9 \leq |A|^{32}$ and $|AA| \leq |A|^{14/9}$, or else there is nothing to prove. Therefore $|A||A + \lambda A||AA| = |A|(|A + \lambda A|^{18}|AA|^9)^{1/18}|AA|^{1/2} \leq |A|^{64/18} \leq p^2 $, as required.

It follows that
\[
|A + \lambda A|^{18}|AA|^9 \gtrsim |A|^{32}.
\]
\end{proof}

\begin{note} 
Using methods from \cite{S_sumsets} one can show that $|A + \lambda A| \leq M|A|$ implies that $|AA| \gtrsim M^{-C} |A|^{58/37}$, where $C>0$ is an absolute constant.  Indeed, by the previous arguments the inequality $|A + \lambda A| \leq M|A|$ gives us the required bound for the third--energy and the common energy, see Lemma \ref{soly2} and Lemma \ref{energy}, and exactly such kind of statements 
are needed in the  eigenvalues  approach from \cite{S_sumsets}.  
Nevertheless, the methods in \cite{S_sumsets} are less elementary (although give slightly better bounds), so we prefer to give a more transparent proof in the current paper. 
\end{note}

\subsection{An improved bound for $|(A-A)(A-A)|$ in terms of $|A|$} In this subsection, we use the previous result to give another result which breaks the $3/2$ threshold, as follows.

\begin{thm} \label{thm:(A+A)(A+A)} Let $A \subset \mathbb F$ and $\lambda \in \F^*$, suppose $|A| \leq p^{9/16}$.
Then 
\[
|(A+\lambda A)(A+\lambda A)| \gtrsim |A|^{\frac{3}{2}+\frac{1}{90}}.
\]
\end{thm}

Before proving Theorem~\ref{thm:(A+A)(A+A)}, we need the following application of Theorem \ref{SdZgen}. See \cite[Corollary 10]{SdZ}.

\begin{lem} \label{thm:A(B+C)}
Let $A,B,C \subset \mathbb F$, suppose $|A||B||C| \leq p^2$. Then
\[
|A(B+C)| \gg (|A||B||C|)^{1/2}.
\]
\end{lem}

\begin{proof}[Proof of Theorem~\ref{thm:(A+A)(A+A)}]
Let $a \in A$ be arbitrary. Suppose that $|A+ \lambda A| \leq |A|^{1+\frac{1}{45}}$. Then $|(A+ \lambda a)+ \lambda (A+ \lambda a)| \leq |A|^{1+\frac{1}{45}}$. By Theorem \ref{thm:fsmp}, which can be applied because of the assumption that $|A| \leq p^{9/16}$, we have
\[
|(A+ \lambda A)(A+ \lambda A)| \geq |(A+ \lambda a)(A+ \lambda a)| \gtrsim |A|^{\frac{14}{9}-\frac{2}{45}}=|A|^{\frac{3}{2}+\frac{1}{90}}.
\]

On the other hand, suppose that $|A+ \lambda A| \geq |A|^{1+\frac{1}{45}}$. Then by Lemma \ref{thm:A(B+C)}
\[
|(A+\lambda A)(A+\lambda A)| \gg |A||A+\lambda A|^{1/2}.
\]
Therefore,
\[
|(A+\lambda A)(A+ \lambda A)| \gg |A|^{\frac{3}{2}+\frac{1}{90}},
\]
as required.

We must verify that we have applied Lemma \ref{thm:A(B+C)} correctly in the case when $\F$ has positive characteristic. That is, we must verify the condition that $|A+\lambda A||A|^2 \leq p$. Note that we may assume that $|A+\lambda A| \leq |A|^{\frac{3}{2}+\frac{1}{90}}$, otherwise $|(A+\lambda A)(A+\lambda A)| \geq |A+\lambda A| \geq |A|^{\frac{3}{2}+\frac{1}{90}}$ and there is nothing to prove. Therefore, 
\[
|A|^2|A+\lambda A| \leq |A|^{\frac{158}{45}} \leq (p^{\frac{9}{16}})^{\frac{158}{45}}<p^2.
\]
\end{proof}

Of particular interest is the case when $\F = \F_p$ and $\lambda = -1$ because it ties well with collinear triples. Bennett, Hart, Iosevich, Pakianathan, and the fourth listed author proved in~\cite{BHIPR2017} that if $A \subset \F_q$ satisfies $|A| \gg q^{2/3}$, then $|(A-A)(A-A)| \gg q$. Their result holds in non-prime order fields. In prime order fields the condition was relaxed to $|A| \gg p^{5/8}$ in~\cite{GPProdDiff}. We offer the following improvement.

\begin{thm}\label{PD} 
Let  $A \subset \F_p$. There exists $a, b \in A$ such that 
\[
|(A-a)(A-b)| \gg \min\left\{ p, \frac{|A|^{5/2}}{p^{1/2}} \right\}
\]
In particular, if $|A| \geq p^{3/5}$, then $\ds |(A-A)(A-A)| \gg p$. 

Moreover, $|(A-A)(A-A)| \gg \min\{p, |A|^{\frac{3}{2} + \frac{1}{90}- \varepsilon} \}$ for all $\varepsilon >0$.
\end{thm}

\begin{proof}
We follow closely an argument of Jones~\cite{Jones2013}, see also~\cite{Roche-Newton2015}. Combining~\eqref{f:T_via_E} and Theorem~\ref{T(A)} gives
\[
\sum_{a,b \in A} \E^\times(A-a, A-b) = \frac{|A|^6}{p} + O(p^{1/2} |A|^{7/2}).
\]
Therefore, by averaging, there exist $a, b \in A$ such that $\E^\times(A-a, A-b) \leq \tfrac{|A|^4}{p} + O(p^{1/2} |A|^{3/2}).$ The Cauchy-Schwarz inequality then gives
\[
|(A-a)(A-b)| \ge \frac{|A|^4}{\E^\times(A-a, A-b)} \gg \min\{p, |A|^{5/2} p^{-1/2}\}.
\]
For the last part we may assume that $|A| \leq p^{3/5}$. Next note that if $|A| \leq p^{9/16}$, then Theorem~\ref{thm:(A+A)(A+A)} with $\lambda =-1$ gives $|(A-A)(A-A)| \gg |A|^{\frac{3}{2} + \frac{1}{90}-\varepsilon}$. When $p^{9/16} \leq |A| \leq p^{3/5}$ we  just proved 
\[
|(A-A)(A-A)| \gg \frac{|A|^{5/2}}{p^{1/2}} \geq |A|^{3/2 + 1/9},
\]
which is asymptotically bigger.
\end{proof}

It is straightforward to modify this argument to obtain the same lower bound for $(A+A)(A+A)$. By considering the number collinear triples $T(A \cup -A)$, it follows that there exist $a,b \in A$ such that $E^\times(A+a,A+b) \ll |A|^4/p + p^{1/2}|A|^{3/2}$, and the rest of the proof is almost identical.

\subsection{Four-fold products of differences}

Theorem~\ref{PD} states that if $A \subset \F_p$ has cardinality $|A| \gg p^{3/5}$ then $(A-A)(A-A)$ contains a positive proportion of the elements of $\F_p$. A complementary question is to seek conditions on $|A|$ that guarantee that $(A-A)(A-A) = \F_p$. The strategy followed in the proof of Theorem~\ref{PD} cannot work in this context. Hart, Iosevich, and Solymosi used additive character sum estimates in $\F_q^2$ to show that if $A \subset \F_q$ satisfies $|A| \gg q^{3/4}$, then $(A-A)(A-A) = \F_q$~\cite{HIS2007}. Note here that their results hold in non-prime order finite fields too. They also obtained conditions for $d$-fold products of differences $\displaystyle \underbrace{(A-A) \dots (A-A)}_{d \textrm{ times}}$ to equal the whole of $\F_q$ for $d \geq 2$. Balog~\cite{Balog2013} improved their bounds for sufficiently large $d$ in $\F_q$.

We combine the collinear triples bound of Theorem~\ref{T(A)} with standard multiplicative character arguments to improve the state of the art for 4-fold products of differences in prime order finite fields only.

\begin{thm}\label{4-fold}
Let $p$ be a prime and $A \subset \F_p$. If $|A| \gg p^{3/5}$, then
\[
(A-A)(A-A)(A-A)(A-A) = \F_p.
\]
\end{thm}

\begin{proof}
Note that $ 0 \in (A-A)(A-A)(A-A)(A-A)$. 

Let $\lambda \in \F_p^*$ and set $r(\lambda):=r_{(A-A)(A-A)(A-A)(A-A)} (\lambda)$ be the number of ways $\lambda$ can be expressed as a product $\lambda = (a_1-b_1)(a_2-b_2)(a_3-b_3)(a_4-b_4)$ with all the $a_i, b_i \in A$. The orthogonality of multiplicative characters, which we denote by $\chi$, implies (see~\cite{Iwaniec-Kowalski2004} for details)
\begin{align*}
r(\lambda) 
& = \frac{1}{p-1} \sum_{\chi} \sum_{a_1,b_2,\dots,a_4,b_4 \in A: a_i \neq b_i} \chi\left( \frac{(a_1-b_1)(a_2-b_2)(a_3-b_3)(a_4-b_4)}{\lambda} \right) \\
& = \frac{1}{p-1} \sum_{\chi} \chi(\lambda^{-1}) \sum_{a_1,b_2,\dots,a_4,b_4 \in A : a_i \neq b_i} \chi((a_1-b_1)(a_2-b_2)(a_3-b_3)(a_4-b_4)) \\
& = \frac{1}{p-1} \sum_{\chi} \chi(\lambda^{-1}) \left( \sum_{a_1,b_1 \in A : a_1 \neq b_1} \chi(a_1-b_1) \right)^4.
\end{align*}

The contribution to the sum coming from the principal character $\chi_0$ is $\displaystyle \frac{|A|^4 (|A|-1)^4}{p-1}$. Using the fact that $|\chi(\lambda)|=1$ for all $\chi$ and applying the triangle inequality we get
\begin{align*}
\left| r(\lambda) - \frac{|A|^4 (|A|-1)^4}{p-1} \right|
& \leq \frac{1}{p-1} \sum_{\chi \neq \chi_0}  \left| \sum_{a_1,b_1 \in A : a_1 \neq b_1} \chi(a_1-b_1) \right|^4 \\
& = \frac{1}{p-1} \sum_{\chi \neq \chi_0}\,  \sum_{a_1,\dots,a_4, b_1,\dots, b_4 \in A} \chi\left(\frac{(a_1-b_1)(a_2-b_2)}{(a_3-b_3)(a_4-b_4)} \right). 
\end{align*}

The orthogonality of characters implies that the right side equals the number of non-zero solutions to $(a_1-b_1)(a_2-b_2)= (a_3-b_3)(a_4-b_4) \neq 0$ with all the $a_i,b_i \in A$ minus the $\tfrac{|A|^4 (|A|-1)^4}{p-1}$ contribution that comes from $\chi_0$. A typical application of the Cauchy-Schwarz inequality (see the proof of Lemma~2.7 in~\cite{Trilinear}, for example) bounds the number of non-zero solutions by $|A|^2 T(A)$. 

Combining all the above with the first part of Theorem~\ref{T(A)} yields 
\[
\left| r(\lambda) - \frac{|A|^4 (|A|-1)^4}{p-1} \right| \leq |A|^2\left( \frac{|A|^6}{p} + O(p^{1/2} |A|^{7/2}) \right) - \frac{|A|^4 (|A|-1)^4}{p-1}
\]
Note now that $\ds \tfrac{|A|^4 (|A|-1)^4}{p-1} = \ds \tfrac{|A|^8}{p} + O(\tfrac{|A|^7}{p})$ and so the above becomes
\[
\left| r(\lambda) - \frac{|A|^4 (|A|-1)^4}{p-1} \right| =  O \left(p^{1/2} |A|^{11/2} + \frac{|A|^7}{p}\right) =  O(|A|^{11/2} p^{1/2}).
\]
One more application of $\ds \tfrac{|A|^4 (|A|-1)^4}{p-1} = \ds \tfrac{|A|^8}{p} + O(\tfrac{|A|^7}{p})$ gives
\[
r(\lambda) = \frac{|A|^8}{p}  + O\left(|A|^{11/2} p^{1/2} + \frac{|A|^7}{p}\right) = \frac{|A|^8}{p}  + O(|A|^{11/2} p^{1/2}). 
\]
The first term dominates when $|A| \gg p^{3/5}$, in which range $r(\lambda)>0$ and so $\lambda \in (A-A)(A-A)(A-A)(A-A)$.
\end{proof}

Once again, a small modification of this argument gives the same result for the set $(A+A)(A+A)(A+A)(A+A)$. The only change comes when we need an upper bound for the number of solutions to $(a_1+b_1)(a_2+b_2)= (a_3+b_3)(a_4+b_4) \neq 0$. The number of such solutions is at most $O(|A|^2T(A \cup -A))$.

Note here that using the well-known bound
\[
\left| \sum_{a_1,a_2 \in A} \chi(a_1-a_2) \right| \leq p^{1/2} |A|
\]
for $\chi \neq \chi_0$ one may extend the result for $d$-fold products of differences with $d \geq 4$. Modifying the proof of Theorem~\ref{4-fold} in a straightforward manner gives 
\[
\underbrace{(A-A) \dots (A-A) }_{d \textrm{ times}} = \F_p
\]
whenever $|A| \gg p^{\frac{d-1}{2d-3}} = p^{\frac{1}{2} + \frac{1}{2(2d-3)}}.$ This improves the bounds of Hart, Iosevich, and Solymosi~\cite{HIS2007}, but is worse than the bounds of Balog~\cite{Balog2013} for sufficiently large $d$.

\subsection[Variant of Theorem 2 for ratios and products of shifts]{Variant of Theorem \ref{thm:fsmp} for ratios and products of shifts}

In this subsection, we give another sum-product type estimate in which we present some quantitative progress beyond the $3/2$ threshold, in the form of the following theorem.

\begin{thm} \label{fsmpm} Let $A \subset \mathbb F$. Let $\alpha \in \F_p^*$. Then if, respectively, $|A|\leq p^{13/24}$ and $|A|\leq p^{23/42}$ one has
\begin{align} \label{fsmpm1} |AA|^{13}\left |\frac{A-\alpha}{A-\alpha}\right|^{5} & \gtrsim |A|^{21}, \\ 
|AA|^{23}|(A-\alpha)(A-\alpha)|^9 & \gtrsim |A|^{37}
\end{align}
In particular, if $|AA| = M|A|$ then $\left|\frac{A-\alpha}{A-\alpha}\right|  \gtrsim M^{-13/5} |A|^{8/5}$ and $|(A-\alpha)(A-\alpha)| \gtrsim M^{-23/9} |A|^{14/9}$.
\end{thm}

The proof is similar to the proof of Theorem \ref{thm:fsmp}. First we need an analogue of Lemma \ref{soly2} --- a bound for third energy which is near-optimal.

\begin{lem} \label{soly22} Let $A \subset \mathbb F$ and let $\alpha \in \mathbb F^*$. Suppose $|AA| \leq p^{2/3}$. Then
\[
\E_3^\times(A+\alpha) \ll \frac{|AA|^5\log|A|}{|A|^2}.
\]
\end{lem}

 We have chosen use a slightly simpler argument, using the bound on the number of collinear quadruples in a Cartesian product point set. This accounts for the slightly worse exponents in Theorem \ref{fsmpm}. However we get the same result in the most relevant case when $|AA| \ll |A|$.

\begin{proof} Let $P=(AA \cup -\alpha A) \times (AA \cup -\alpha A)$. The idea is to bound the number of collinear quadruples from above and below, much in the same way as Elekes and Ruzsa \cite{Elekes-Ruzsa2003} did for collinear triples.

Let $d,d' \in A$ be arbitrary and suppose that $\frac{a+\alpha}{a'+\alpha}=\frac{b+\alpha}{b'+\alpha}=\frac{c+\alpha}{c'+\alpha}$ for some $a,\ldots,c' \in A$. Then observe that the four points $(-\alpha d,-\alpha d'),(da,d'a'), (db,d'b'),(dc,d'c')$ are collinear. This implies that the number of collinear quadruples in $P$ is at least 
\[
|A|^2 \E_3^\times (A+\alpha).
\]
Applying Theorem \ref{T(A)} and our assumption that $|AA| \leq p^{2/3}$, we have 
\[
|A|^2 \E_3^\times(A+\alpha) \leq Q(AA \cup -\alpha A) \ll |AA|^5\log|A|.
\]
A rearrangement of this inequality completes the proof.
\end{proof}

We will also use an analogue of Lemma \ref{energy}, implicit in \cite[proof of Proposition 2]{AYMRS}.
\begin{lem} \label{energym} Let $A,B \subset \mathbb F$ such that $|A||AA||B| \leq p^2$ and also $\max(|AA|,|B|)\leq (|A||AA||B|)^{1/2}.$ Then for any  $\alpha \in \F^*$,
\[
\E^\times(A+\alpha,B) \ll \frac{|AA|^{3/2}|B|^{3/2}}{|A|^{1/2}}.
\]
\end{lem}

\begin{proof}[Proof of Theorem \ref{fsmpm}]
The remainder of the proof of Theorem \ref{fsmpm} is very similar to that of Theorem \ref{thm:fsmp}, and so we present only the full proof of inequality \eqref{fsmpm1}. 

Apply respectively \eqref{difs} in the multiplicative setting, Lemma \ref{soly22} and Lemma \ref{energym} as follows:
\begin{align} \label{almost2}
\begin{split}
|A|^8 &\leq \E_3^\times(A+\alpha)\left|\frac{A+\alpha}{A+\alpha}\right|\E^\times\left(A+\alpha,\frac{A+\alpha}{A+\alpha}\right) 
\\ &\ll \frac{|AA|^{5}\left|\frac{A+\alpha}{A+\alpha}\right|\E^\times\left(A+\alpha,\frac{A+\alpha}{A+\alpha}\right) }{|A|^{2}}\log |A| 
\\ &\ll \frac{|AA|^{13/2}\left|\frac{A+\alpha}{A+\alpha}\right|^{5/2}}{|A|^{5/2}}\log |A|.
\end{split}
\end{align}
Rearranging this inequality gives the inequality \eqref{fsmpm1}. However, when applying Lemmata \ref{soly22} and \ref{energym} in the finite characteristic setting, we must ensure that the relevant conditions are met.

We may assume that $|AA|^{13}\left | \frac{A+\alpha}{A+\alpha} \right|^5 \leq |A|^{21}$, or else there is nothing to prove. Therefore, we may assume that $|AA| \leq |A|^{16/13}$ and $\left | \frac{A+\alpha}{A+\alpha} \right| \leq |A|^{8/5}$. The conditions of Lemma \ref{soly22} hold, since $|AA| \leq |A|^{16/13} \leq p^{2/3}$, where the last inequality uses the condition $|A| \leq p^{13/24}$.

The first condition of Lemma \ref{energym} holds, since
\begin{align*}
|A||AA|\left| \frac{A+\alpha}{A+\alpha} \right | & = |A|\left(|AA|^{13}\left| \frac{A+\alpha}{A+\alpha} \right |^5 \right)^{1/13} \left| \frac{A+\alpha}{A+\alpha} \right |^{8/13}
\\& \leq |A|^{1+21/13+64/65}
\\ & \leq p^2,
\end{align*}
where the last inequality uses the condition $|A| \leq p^{13/24}$. Finally, the second condition of Lemma \ref{energym} holds trivially in this application.
\end{proof}

\subsection{A bound for multiplicative energy in terms of $|A+A|$}
To describe the last main result of this section we need some preparations. 
The following theorem may be extracted from \cite{shkredov2013results}.
\begin{thm}
\label{thm:eigenval}
Let $A$ be a finite subset of an abelian group $G$.
Suppose there are parameters $D_1$ and $D_2$ such that
\[
\E_3(A) \leq D_1|A|^3
\]
and for any finite set $B\subset G$
\[
\E(A,B) \leq D_2|A||B|^{3/2}.
\]
Then
\[
\E(A) \ll D_1^{6/13}D_2^{2/13}|A|^{32/13} \log^{12/13}|A|.
\]
\end{thm}
This implies the following bound for the multiplicative energy of a subset of $\F$ with small sumset.
It is implicit in the proof of Theorem~\ref{thm:eigenval} that the bound for $\E(A,B)$ only needs to hold for $|B|\leq 4|A|^4/\E(A)$.
\begin{thm}
\label{thm:multEnergyBnd}
Suppose $A$ is a subset of $\F_p$ such that $|A+A|= M|A|$ and $|A|\leq p^{13/23}M^{25/92}$.
Then
\[
\E^\times(A) \lesssim M^{51/26}|A|^{32/13}.
\]
\end{thm}
Note that $32/13 = 5/2 - 1/26$, so for $M$ sufficiently small, this improves the bound $\E^\times(A)\ll M^{3/2}|A|^{5/2}$ provided by Lemma~\ref{energy}.
\begin{proof}
We will apply Theorem \ref{thm:eigenval} where $G$ is the multiplicative group $\F^*$.
Thus we require bounds on the third multiplicative energy of $A$ and the multiplicative energy of $A$ with another set $B$.

By Lemma~\ref{energy} we have
\[
\E^\times(A,B)\ll \frac{|A+A|^{3/2}}{|A|^{1/2}}|B|^{3/2},
\]
so we  may take $D_2=M^{3/2}$ whenever $M|A|^2|B|\leq p^2$.
Since we may assume that $|B|\leq 4|A|^4/\E^\times (A)$, and the theorem is trivial if $\E^\times (A) < 4 M^{51/26}|A|^{32/13}$, without loss of generality,
\[
|B| \leq \frac{|A|^{20/13}}{M^{51/26}}, 
\]
hence $M|A|^2|B|\leq p^2$ whenever
\[
\frac{|A|^{46/13}}{M^{25/26}}\leq p^2,
\]
which follows from the assumption $|A|\leq p^{13/23}M^{25/92}$.

By Lemma~\ref{soly2}
\[
\E^\times_3(A) \ll \frac{|A+A|^{15/4}}{|A|^{3/4}}\log|A|,
\]
whenever $|A|<p^{2/3}/M^{1/12}$, so we may take $D_1=M^{15/4}\log|A|$ in this range.

The desired bound follows from plugging $D_1$ and $D_2$ into Theorem \ref{thm:eigenval}.

To check the constraint, first note that $M^{51/26}{|A|^{32/13}}>M^{3/2}|A|^{5/2}$ unless $M < |A|^{1/12}$.
For $M<|A|^{1/12}$, we have
\[
p^{13/23}M^{25/92} \leq \frac{p^{2/3}}{M^{1/12}}.
\]
\end{proof}

\begin{cor}
\label{cor:T_first}
Let $A\subset \F,$ such that $|A+A|=M|A|$ and 
$|A|\leq p^{13/23}M^{25/92}$. Then
\[
T(A) \lesssim M^{\frac{51}{26}}|A|^{\frac 92 - \frac 1{26}}.
\]
\end{cor}
\begin{proof}
We use the expression
\[
T(A) = \sum_{a,b\in A} \E^\times(A+a,A+b).
\]
By Cauchy-Schwarz, $\E^\times(A+a,A+b)\leq \E^\times(A+a)^{1/2} \E^\times(A+b)^{1/2}$, so 
\[
T(A)\leq \left( \sum_{a\in A} \E^\times(A+a)^{1/2}\right)^2.
\]
By a second application of Cauchy-Schwarz, we have
\[
T(A)\leq |A|\sum_{a\in A} \E^\times(A+a).
\]

Now, since $|A+A|=M|A|$, we have $|(A+a)+(A+a)|=M|A|$, so by Theorem \ref{thm:multEnergyBnd}
\[
\E^\times(A+a) \lesssim  M^{51/26}|A|^{32/13}.
\]
Thus
\[
T(A) \lesssim |A|^2 \cdot M^{51/26}|A|^{32/13},
\]
which is the claimed bound.
\end{proof}

Let us obtain a ``dual'' version of the results above.

\begin{thm}
\label{thm:multEnergyBnd'}
Suppose $A\subset\F$ such that  $|AA|= M|A|$ and $|A|\leq p^{13/23}M^{10/23}$.
Then for any $\alpha \in \F^*$ one has 
\[
\E^\times(A+\alpha) \lesssim M^{33/13}|A|^{32/13}.
\]
\end{thm}
Note that $32/13 = 5/2 - 1/26$, so for $M$ sufficiently small, this improves the bound $\E^\times(A)\ll M^{3/2}|A|^{5/2}$ provided by Lemma~\ref{energym}.
\begin{proof} 
By Lemma~\ref{energym} we have
\[
\E^\times(A+\alpha,B)\ll \frac{|AA|^{3/2}}{|A|^{1/2}}|B|^{3/2},
\]
so we  may take $D_2=M^{3/2}$ whenever $M|A|^2|B|\leq p^2$.
Since we may assume that $|B|\leq 4|A|^4/\E^\times(A+\alpha)$, and the theorem is trivial if $\E^\times (A+\alpha) < 4 M^{33/13}|A|^{32/13}$, without loss of generality,
\[
|B| \leq \frac{|A|^{20/13}}{M^{33/13}}, 
\]
hence $M|A|^2|B|\leq p^2$ whenever
\[
\frac{|A|^{46/13}}{M^{20/13}}\leq p^2,
\]
which reduces to the assumption $|A|\leq p^{13/23}M^{10/23}$.

By Lemma~\ref{soly22}
\[
\E^\times_3(A+\alpha) \ll \frac{|AA|^{5}}{|A|^{2}}\log|A|,
\]
whenever $|A|<p^{2/3}/M^{}$, so we may take $D_1=M^{5}\log|A|$ in this range.

The desired bound follows from plugging $D_1$ and $D_2$ into Theorem \ref{thm:eigenval}.

To check the constraint, first note that $M^{33/13}{|A|^{32/13}}>M^{3/2}|A|^{5/2}$ unless $M<|A|^{1/27}$.
For $M<|A|^{1/27}$, we have
\[
p^{13/23}M^{10/23} \leq \frac{p^{2/3}}{M^{}}.
\]
\end{proof}

Repeating the arguments of the proof of Corollary \ref{cor:T_first}, we obtain

\begin{cor}\label{cor:T_second}
Let $A$ be a subset of $\F$ such that $|AA|=M|A|$, 
$|A|\leq p^{13/23}M^{10/23}$. 
Then
\[
T(A) \lesssim M^{\frac{33}{13}}|A|^{\frac 92 - \frac 1{26}}.
\]
\end{cor}

 \section*{Acknowledgment}

We thank the referee for helpful suggestions, and in particular for a way to improve the bounds and exposition in Theorem~\ref{thm:proddiff}. We also thank Igor Shparlinski and Sophie Stevens for helpful conversations. The third and fourth authors are grateful to the Institute for Pure and Applied Mathematics (IPAM), supported by the National Science Foundation, for having given them an opportunity to work on this paper during the CCG2016 Reunion Conference.

\appendix

\section{Discussion on collinear triples and quadruples}
\label{sec:discussTAQA}
In this section we elaborate on the some of the remarks made in Section~\ref{sub_section_collinear}.

First, we justify the formula \eqref{Ti Qi} for $T(A)$ and $Q(A)$ in terms of the incidence function:
\[
T(A) = \sum_{\ell : i(\ell) >1} i(\ell)^3 + O(|A|^4) \quad \text{ and } \quad Q(A) = \sum_{\ell : i(\ell) >1} i(\ell)^4 + O(|A|^4).
\]
Let us first justify the expression for $T(A)$. We may suppose that $|A| >1$ and so $A \times A$ is not a singleton because for singleton $A$ the number of collinear triples is $O(|A|^4)$. Any collinear triple with at least two distinct points from $A \times A$ determines a unique line $\ell$ with $i(\ell)>1$ and so every such triple is counted exactly once in the third moment of $i(\ell)$. We are left with collinear triples of the form $(u,u,u)$. Any such triple may be counted in the third moment more than once. The total number of times it is counted equals the number of lines with $i(\ell)>1$ that $(u,u,u)$ is incident to. Since every line with $i(\ell)>1$ determines at least one distinct point of $(A \times A) \setminus \{u\}$, we see that the contribution to the third moment coming from $(u,u,u)$ is $O(|A|^2)$. There are $|A|^2$ collinear triples $(u,u,u)$ so the difference between $T(A)$ and the third moment is $O(|A|^4)$. A similar argument works for $Q(A)$.

Next we justify the alternative expressions for $T(A)$ and $Q(A)$ in terms ratios of differences of elements of $A$.
Recall that we defined $t(x)$ and $q(x,y)$ by
\[
t(x) = |\{ (a,b,c) \in A \times A \times A  \colon x = \tfrac{b-a}{c-a} \}| 
\]
and
\[
  q(x,y)=|\{(a,b,c,d)\in A \times A \times A \times A \colon \tfrac{b-a}{c-a} = x \, , \, \tfrac{d-a}{c-a} = y \}|.
\]
We claim in \eqref{t expression} and \eqref{q expression} that
\[
T(A) = \sum_{x \in \F_p} t(x)^2 + |A|^4 \qquad\mbox{and}\qquad
Q(A) = \sum_{x,y \in \F_p} q(x,y)^2 + |A|^5.
\]
We verify only the second identity as it is probably the most complicated of the two. The sum $ \sum q(x,y)^2$ is the number of solutions to
\[
\frac{b-a}{c-a} = \frac{b'-a'}{c'-a'} \quad \text{ and } \quad  \frac{d-a}{c-a} = \frac{d'-a'}{c'-a'} \;, \; a, b , \dots, d' \in A
\] 
Fixing $(a,a')$, the number of solutions coming from the first equality equals the number of ordered pairs $(b,b'), (c,c')$ that are on a non-vertical line incident to $(a,a')$ (this is because $(a,a') - (b,b')$ and $(a,a') - (c,c')$ have the same slope). Similarly, the number of solutions coming from the second equality equals the number of ordered pairs $(c,c')$ and $(d,d')$ that are on a non-vertical line incident to $(a,a')$. Therefore the sum $\sum q(x,y)^2$ counts the number of collinear quadruples on non-vertical lines. There are $|A|$ vertical lines, each incident to $|A|$ points of $A \times A$ and identity~\eqref{q expression} follows.

Next, we elaborate on what is known for collinear triples. The contribution to collinear triples coming from the $|A|$ horizontal lines incident to $|A|$ points in $A \times A$ is $|A|^4$. All other collinear triples can be counted by the number of solutions to
\begin{equation}\label{Alg T(A)}
(b-a) (c'-a') = (b'-a')(c-a) \;, \; a, b , \dots, d' \in A.
\end{equation}
It follows from this that the expected number of (solutions to~\eqref{Alg T(A)} and therefore of) collinear triples in $A \times A$ where $A$ is a random set (where elements of $\F$ belong to $A$ independently with probability $|A|/p$) is $\ds \frac{|A|^6}{p} + O(|A|^4)$. There are $O(|A|^4)$ solutions to~\eqref{Alg T(A)} where the products are zero. For each 5-tuple $(a,b,a',b',c')$ that gives a non-zero solution to~\eqref{Alg T(A)} there is a unique element $c \in \F$ that satisfies~\eqref{Alg T(A)} and it belongs to $A$ with probability $|A|/p$. There are $|A|^5 + O(|A|^4)$ such $5$-tuples and the claim follows.

Another interesting example is that of sufficiently small arithmetic progressions. First note that in general $T(A)$ equals 
\[
\sum_{a,a' \in A} \sum_{x} f_{a,a'}^2(x) + O(|A|^4),
\]
where $f_{a,a'}(x)$ is the number of ways one can express $x$ as a product $(b-a)(c'-a')$ with $b,c' \in A$. The support of $f_{a,a'}$ is the set $(A-a)(A-a') \subset (A-A)(A-A)$.

Observe that for all $a, a' \in A$, one has $\ds \sum_{x} f_{a,a'}(x) = |A|^2$ because each pair $(b,c') \in A \times A$ contributes 1 to the sum. 

Applying the Cauchy-Schwarz inequality gives
\[
T(A) = \sum_{a, a' \in A} \sum_{x} f_{a,a'}^2(x) \geq \sum_{a,a' \in A} \f{\left( \sum_{x} f_{a,a'}(x) \right)^2}{|\supp(f_{a,a'})|} \geq \f{|A|^6}{|(A-A)(A-A)|}.
\]
Now take $A = \{1,\dots, \sqrt{p}/3\} \subset \mathbb{Z}$. For all $a,b,a', c'\in A$ the product $(b-a)(c'-a')$ belongs to $\{-p/2,\dots , p/2\}$. This means that $|(A-A)(A-A)|$ is the same whether $A$ is taken to be a subset of $\mathbb{Z}$ or of $\F_p$. Ford has shown in~\cite{Ford2008} that $|(A-A)(A-A)|$ (in $\mathbb{Z}$ and hence) in $\F_p$ is $O(|A|^2 / \log^\gamma|A|)$ for some absolute constant $\gamma < 1$. Substituting in the lower bound  $$T(A) \geq |A|^6 / |(A-A)(A-A)|$$ implies that for $A = \{1,\dots, \sqrt{p}/3\} \subset \F_p$, we have $T(A) \gg |A|^4 \log^\gamma|A|$.

Over the reals, Elekes and Ruzsa observed in~\cite{Elekes-Ruzsa2003} that the Szemer\'edi-Trotter point-line incidence theorem~\cite{Szemeredi-Trotter1983} implies
\[
T(A) \ll |A|^4\log|A| .
\]
Because of this and the two examples discussed above, it is possible that the inequality
\[
T(A) \ll \frac{|A|^6}{p} +  |A|^4\log|A|
\]
is correct up to logarithmic factors in $\F_p^2$.

We conclude with a note on the relationship between cross ratios and $R[A]$.
\begin{note}
\label{note:T_C}
As in subsection \ref{sub_section_collinear} the set $C[A]$ is connected with the number of the solutions of the equation
\[
\E_{C[A]}:= \left| \left\{ \frac{(a-b)(c-d)}{(a-c)(b-d)} = \frac{(a'-b')(c'-d')}{(a'-c')(b'-d')} :\, a,b,c,d, a',b',c',d' \in A \right\} \right| \,.
\]
Write
\[
\E_{R[A,B]}= \left| \left\{ \frac{a_1-a}{a_2-a} = \frac{b_1-b}{b_2-b} :\, a_1,a_2,a \in A,\, b_1,b_2,b \in B \right\} \right| \,.
\]
Then, 
\[
	\E_{C[A]} = \sum_{a,b\in A} \E_{R[(A-a)^{-1}, (A-b)^{-1}]} \,.
\]
In particular, for $|A|=|B|\leq p^{2/3}$, we have, by Cauchy-Schwarz and the estimates in
~\cite[proof of Proposition~5]{AYMRS} that $\E_{R[A,B]}\ll|A|^{9/2}$ and hence
$\E_{C[A]} \ll \sum_{a,b\in A} |A|^{9/2} \ll |A|^{13/2}$.

Finally, as is the case with $R[A]$, the set $C[A]$ has the property that $C[A]=C[A]^{-1}$ and $C[A]=1-C[A]$ (the last identity can be obtained directly or via the identity $C[A] = \bigcup_{a\in A} R[(A-a)^{-1}]$). 
\end{note}

\section[Proof of Lemma 21]{Proof of Lemma \ref{soly2}}
\label{sec:soly2proof}
The general strategy is as follows: if $m$ is an element of $A /A$, then the line through the origin with slope $m$ intersects $A\times A$ in $r_{A / A}(m)$ points.
We take a collection of $M$ lines through the origin with slope $m$ where $r_{A / A}(m)\approx k$ for some $k>0$ so that $M\, k^3$ is roughly $\E^\times_3(A)$. We then form a larger collection of lines by translating these $M$ lines by each point $(a,a')$ in $A\times A$ to form a set of lines $L$.
On one hand, each line in $L$ intersects $(A + \lambda A)\times (A + \lambda A)$ in at least $k$ points.
On the other hand, each point of $A\times A$ is incident to $M$ lines in $L$.
In each case we apply an incidence bound, and this allows us to bound $M\, k^3$.

This is the same strategy employed in \cite{S2}, except that we employ the incidence bounds from Lemmata \ref{BKT In} and \ref{L_M} instead of the Szemer\'edi-Trotter theorem.

\begin{proof}[Proof of Lemma \ref{soly2}]

Let 
\[
P:=\left\{m \in A / A:r^2_{A / A}(m)\geq \frac{\E_3^\times(A)}{2|A|^2}\right\}.
\]
By a popularity argument, we have
\[
\sum_{m\in P}r_{A /A}^3(m) \geq\frac 12 \E^\times_3(A).
\]

After a dyadic decomposition, it follows that there is an integer $k^2\geq \E_3^\times(A)/2|A|^2$ and a set $R_k \subset P$ such that for all $m \in R_k$, $k \leq r_{A /A}(m)< 2k$ and with
\[|R_k|k^3 \gg \frac{\E_3^\times(A)}{\log |A|}.\]
A trivial observation is that $k \leq |A|$, since $r_{A /A}(x) \leq |A|$ for all $x$. Note that we can assume that it is not the case that $|R_k| \ll 1$. Indeed, if this were to be the case, then
\[
\E_3^\times(A) \ll |R_k|k^3\log|A| \ll |A|^3\log|A| \leq \frac{|A + \lambda A|^{15/4}}{|A|^{3/4}}\log |A|
\]
and there is nothing to prove.

Consider the point set $A \times  A$ in the plane. Construct a set of lines $L$ such that every point in $A \times  A$ is supported on a line with slope $m$ for every $m \in R_k$. Since every point $p \in A \times  A$ has $|R_k|$ lines from $L$ passing through it, it follows from a variation of Theorem \ref{SdZgen} (see \cite[Theorem 7]{murphy2017second}) that
\begin{equation}\label{f:4_terms_tmp}
|A|^2|R_k| \leq  \mathcal{I}(A \times  A,L) \ll \frac{|A|^{3/2}|L|}{p^{1/2}} +|A|^{5/4} |L|^{3/4} + |A|^2 + |L|.
\end{equation}

If $|A| > p^{1/3}$, then we have
\begin{equation}
|A|^2|R_k| \leq  \mathcal{I}(A \times A,L) \ll \frac{|A|^{3/2}|L|}{p^{1/2}} +|A|^{5/4} |L|^{3/4}.
\label{ince1}
\end{equation}
Now consider the point set $P=(A + \lambda A) \times (A + \lambda A)$. Each line in $l \in L$ contains at least $k$ points from $P$. Indeed, such a line has gradient $m \in R_k$ with $m=\frac{b_1}{a_1}=\frac{b_2}{a_2}=\cdots=\frac{b_k}{a_k}$, and contains a point $(a,b) \in A \times A$. Therefore
\[
(a+\lambda a_1,b+\lambda b_1), \dots ,(a+\lambda a_k,b+\lambda b_k) \in l \cap P.
\]

To conclude the argument, we will consider three cases:
\begin{enumerate}
\item $k < 2|A + \lambda A|^2/p$
\item $2|A + \lambda A|^2/p\leq k < 2|A + \lambda A|^{3/2}/p^{1/2}$
\item $k \geq 2|A + \lambda A|^{3/2}/p^{1/2}$.
\end{enumerate}

\begin{case}
The bound $k^2\geq \E_3^\times(A)/2|A|^2$ implies
\[
\E_3^\times(A) \leq \frac{8|A|^2|A + \lambda A|^4}{p^2}.
\]
If $|A|^{11}|A + \lambda A|\leq p^8$, then
\[
\frac{8|A|^2|A + \lambda A|^4}{p^2} \ll \frac{|A + \lambda A|^{15/4}}{|A|^{3/4}},
\]
as required.
\end{case}

\begin{case}
For case 2, we apply Lemma~\ref{BKT In} to $P=(A + \lambda A)\times (A + \lambda A)$ and $L$ to conclude that
\begin{equation}
|L| \leq \frac{4p|A + \lambda A|^2}{k^2}.
\label{case2bd}
\end{equation}
If $|A|\leq p^{1/3}$, then applying first the upper bound on $k$ and then $|A| \leq p^{1/3}$ gives
\[
\E^\times_3(A) \leq 2|A|^2k^2 \ll \frac{|A + \lambda A|^3 |A|^2}{p} \leq \frac{|A+\lambda A|^3}{|A|} \leq \frac{|A + \lambda A|^{15/4}}{|A|^{3/4}}.
\]
If $|A|>p^{1/3}$, then by \eqref{ince1} and \eqref{case2bd}, we have
\begin{align*}
|A|^2|R_k| \leq \mathcal{I}(A\times A, L) 
&\ll \frac{|A|^{3/2}|L|}{p^{1/2}} + |A|^{5/4}|L|^{3/4}\\
&\ll \frac{p^{1/2}|A|^{3/2}|A + \lambda A|^2}{k^2} + \frac{p^{3/4}|A|^{5/4}|A + \lambda A|^{3/2}}{k^{3/2}}.
\end{align*}
Using $k \ll |A + \lambda A|^{3/2}/p^{1/2}$ we have
\[
|A|^2|R_k| \ll \frac{|A|^{3/2}|A + \lambda A|^{7/2}}{k^3} + \frac{|A|^{5/4}|A + \lambda A|^{15/4}}{k^3}.
\]
If the second term dominates, then $k^3|R_k|\ll |A + \lambda A|^{15/4}/|A|^{3/4}$.
If the first term dominates, then
\[
k^3|R_k| \ll \frac{|A + \lambda A|^{7/2}}{|A|^{1/2}} \leq \frac{|A + \lambda A|^{15/4}}{|A|^{3/4}}.
\]
Thus both terms are acceptable because $\E_3^\times(A) \ll |R_k|k^3 \log|A|$.
\end{case}

\begin{case}
For case 3, we apply Lemma~\ref{L_M} to $P=(A + \lambda A)\times (A + \lambda A)$ and $L$ to deduce that
\begin{equation}
|L| \ll\frac{|A + \lambda A|^5}{k^4}.
\label{Lbound}
\end{equation}
Note also that \eqref{case2bd} holds again in this case. 

By \eqref{f:4_terms_tmp}, \eqref{Lbound} and the assumption of case 3, we have
\begin{align*}
|A|^2|R_k| & \ll  \frac{|A|^{3/2}|L|}{p^{1/2}} +|A|^{5/4} |L|^{3/4} + |L|. 
\\&\ll \frac{|A|^{5/4}|A + \lambda A|^{15/4}}{k^3} + \frac{|A + \lambda A|^5}{k^4}
\numberthis \label{case3bd}
\end{align*}
If the first term in \eqref{case3bd} dominates then
\[
\E_3^{\times}(A) \ll k^3 |R_k| \log|A| \ll \frac{|A + \lambda A|^{15/4}}{|A|^{3/4}}\log |A|.
\]
If the second term in \eqref{case3bd} dominates then
\[
	\frac{\E^\times_3 (A)}{\log |A|} \ll k^3|R_k| \ll \frac{|A + \lambda A|^5}{k|A|^2} \ll \frac{|A + \lambda A|^5}{(\E^\times_3 (A))^{1/2} |A|},
\]
where the last inequality follows from the bound $k^2 \geq \E_3^\times(A) /2|A|^2$.
This gives a better bound for $\E^\times_3 (A)$, and thus completes the proof.
\end{case}
\end{proof}

\begin{note}
	A shorter way of proving of a slightly weaker version of Lemma \ref{soly2} has already come up in the main body of the paper as the proof of Lemma \ref{soly22}. One can just use formula (\ref{f:Q_via_E}) and notice that for any $a\in A$ the sumset $A+A$ contains $A+a$, which yields the slightly weaker estimate (see also Lemma \ref{soly22} below), namely, 
    \[
    \E_3^\times(A) \ll \frac{|A+A|^5\log|A|}{|A|^2}.
    \]
  In the third energy bounds of Lemmata \ref{soly2}, \ref{soly22}, the $\log|A|$ term appears essentially owing to the bounds for the quantity $Q(A)$ in Theorem \ref{T(A)}.
\end{note}

\section[Additional arguments for Theorem 4]{Additional arguments for Theorem~\ref{arr}}
\label{sec:arrAdditional}

In this appendix we present the proof of Theorem~\ref{arr}, with estimates as claimed, using the lemmata from Section \ref{Misha} and a slightly more involved pigeonholing arguments.
Similar to the proof in Section \ref{Misha} we assume that $P$ defines sufficiently many directions for their cross-ratios to be well-defined.   

Let the set $P_2$ contain all points of $P$, which are supported on lines through the origin with at least $w$ (to be chosen) points per line, and $P_1=P\setminus P_2$. One of the two sets $P_1,\,P_2$ contains at least half of the elements of $P$. We will later apply Lemma \ref{both} to $P_1$ and now focus on $P_2$.  

\medskip
The rest of the proof is about $P_2$, so let us just write $T=\omega(P_2)\subset \F^*$ (rather than $\F$, by definition the set $\omega(P)$ excludes zero). We hope that the reader will not be hindered by the fact that the notations $T$ and further $Q,E$ throughout the rest of this appendix are local and do not relate to what $T$, etc. stood for in earlier parts of this paper.

Consider the equation
\begin{equation} t_1t_2 - t_3t_4 = t_1't_2' - t_3't_4':\; t_1,\ldots, t_4'\in T,
\label{eqng}\end{equation}
and let $Q$ denote the number of solutions. Lemma \ref{omega_upper_bnds}  gives the upper bound \eqref{omega_second_moment}:
\begin{equation}\label{upper}Q=O(|T|^{13/2}).\end{equation}
This bound is conditional on $|T|\leq p^{2/3}$, which we assume just as we did when the  lemma was applied in Section \ref{Misha}. 

Consider the equation \eqref{t_eq} on $T$, let us restate it:
\begin{equation}\label{teq}
t_1t_2=t_3t_4-t_5t_6:\;t_1,\ldots,t_6\in T.
\end{equation}

Let us also consider the equation
\begin{equation}\label{teqs}
s=t_3t_4-t_5t_6:\;t_3,\ldots,t_6\in T,\,s\in TT.
\end{equation}
We will use Lemmata \ref{qequiv}, \ref{phi} to get a lower bound  for the number of solutions of equation \eqref{teqs} in the following Lemma \ref{clue},  to be then compared to the standard upper bound by Cauchy-Schwarz, which will involve the bound \eqref{upper}.

\begin{lem}\label{clue}
Suppose $|P_2|/w\leq p^{5/12}$. There exist $i,j\geq 0,$ with $2^{i-j+1}w\geq 1$, and such that  equation \eqref{teqs}
has 
\[
\gtrsim  |P_2|^{8/5}w^{7/5} 2^{\frac{1}{2}i+j}
\]solutions, with $s\in S \subset TT$, such that every member of $S$ has $\Omega(2^{i-j}w)$ realisations as a product of two elements from $T$.
\end{lem}
 
 \begin{proof}
Similar to Section \ref{Misha} we deal with equivalence classes of quadruples $(a,b,c,d)$ of points in $P_2$, supported in four distinct directions $\delta_a,\delta_b,\delta_c,\delta_d$ and the map $\Phi$ described by Lemma \ref{phi}.

By   Lemma \ref{qequiv},  if we have  two point quadruples $(a,b,c,d)$ and $(a',b',c',d')$ from $P_2$, the corresponding quadruples $(t_{ab},t_{cd}, t_{ac},t_{bd})$ and $(t_{a'b'},t_{c'd'}, t_{a'c'},t_{b'd'})$
will not be the same, as long as $[\delta_a,\delta_b,\delta_c,\delta_d] \neq [\delta_{a'},\delta_{b'},\delta_{c'},\delta_{d'}].$

So let us fix a direction quadruple $(\delta_a,\delta_b,\delta_c,\delta_d)$, representative of a particular cross-ratio value and allow the points $a,b,c,d$ to vary along the lines in these fixed directions, respectively. By Lemma \ref{phi}, an equivalence class of the quadruple of points $(a,b,c,d)$ supported in the corresponding directions arises by  scaling $a,d$ by some factor $x\neq 0$ and $b,c$ simultaneously by $x^{-1}$.

Let us use dyadic partitioning. We have chosen to do it rather carefully in order to show that the proof of Theorem \ref{arr} itself does not incur additional logarithmic factors, but for the one from Lemma \ref{crt}. The reader who is not willing to bother about this may just assume that there are two popular dyadic groups being chosen in the ensuing two dyadic pigeonholing arguments, in the usual way.

 By the $i$th dyadic group of lines, as usual, we mean a group of all lines through the origin, each one of which, for $i=0,1,\ldots, \lceil\log_2|P_2|\rceil$, supports a number of points of $P_2$ in the interval $[2^iw, 2^{i+1}w)$. Let $\eta<1$ be sufficiently close to 1, say $\eta=15/16$. Observe that there exists an $i$th dyadic group of lines through the origin  such that the number  $N_i$ of points of $P_2$ supported on the lines in this dyadic group is $ \Omega(\eta^{i}|P_2|)$. 
Indeed, by the pigeonhole principle,  if each $N_i< \epsilon\, \eta^{i} |P_2|$, for some sufficiently small constant $\epsilon(\eta)$, the quantities $N_i$ would  sum to merely a fraction of $|P_2|$.

From now on we restrict the consideration to the part of $P_2$ supported on the above popular dyadic group of lines, that is some $i$ such that $\,N_i= \Omega(\eta^{i}|P_2|)$. Indeed there are $|D|=\Omega( 2^{-i} \eta^i |P_2|/ w)$ members in the set of directions $D$ defined by lines in this group. Hence, since by the conditions of the lemma $|D|\leq p^{5/12}$, we can find, by Lemma \ref{crt}, $\gtrsim |D|^{8/5}$ distinct  quadruples of popular lines, whose directions $(\delta_a,\delta_b,\delta_c,\delta_d)$ yield distinct cross-ratios.

Thus  this $i$th group of popular lines contributes the number $E$ of solutions of \eqref{teq}, with
\[
E\gtrsim |P_2|^{8/5}(2^iw)^{12/5} \eta^{\frac{8}{5}i}.
\]

Since $\eta$ is sufficiently close to $1$,  the worst possible case is $i=0$. Moreover, since $\eta$ sufficiently close to $1$ we can subsume it by slightly reducing the powers of $2$, and simplify the latter bound to 

\begin{equation}
E \gtrsim |P_2|^{8/5} w^{12/5}\, 2^{2i}.
\label{is}\end{equation}
solutions of \eqref{teq}.

We now count the corresponding solutions $(s,t_{ab},t_{cd},t_{ac},t_{bd})$  of equation \eqref{teqs}. Once the direction quadruple $(\delta_a,\delta_b,\delta_c,\delta_d)$ has been fixed, inequivalent quadruples $(a,b,c,d)$ that give rise to $(\delta_a,\delta_b,\delta_c,\delta_d)$, result in distinct solutions of equation \eqref{teqs}. For a fixed direction quadruple $(\delta_a,\delta_b,\delta_c,\delta_d)$, the number of solutions of \eqref{teqs} we get is the number of equivalence classes of point quadruples $(a,b,c,d)$, supported in these directions. Since we have isolated ourselves to a dyadic set of lines with $\Theta(2^iw)$ points per line, there are $O(2^iw)$ elements in each equivalence class.

We now perform another dyadic partitioning, by the number of elements in an equivalence class of a quadruple $(a,b,c,d)$ and are going to choose a popular dyadic value of the number of elements in a class over the set of $\gtrsim |D|^{8/5}$ cross-ratio-representative direction quadruples. Fix $j\geq 0$ and such that $2^{-j+i}w\geq 1$.
For each cross-ratio-representative direction quadruple $(\delta_a,\delta_b,\delta_c,\delta_d)$, we partition the set of point quadruples $(a,b,c,d)$ lying in the four directions, so that the $j$th dyadic equivalence class of $(a,b,c,d)$ has a number of members lying in the interval $[2^{i-j}w,2^{i-j+1}w).$

Thus each equivalence class provides a single distinct solution $(s,t_{ab},t_{cd},t_{ac},t_{bd})$ of \eqref{teqs}, arising from  the set of sextuples $(t_{ad},t_{bc},t_{ab},t_{cd},t_{ac},t_{bd})$ solving \eqref{teq}, with the number of realisations in the $j$th dyadic group lying in the interval $[2^{i-j}w,2^{i-j+1}w).$

We now ``unfix'' the direction quadruple  $(\delta_a,\delta_b,\delta_c,\delta_d)$ representing a cross-ratio and from the union of the equivalence classes of $(a,b,c,d)$ over the direction quadruples  $(\delta_a,\delta_b,\delta_c,\delta_d)$, representative of the set of  cross-ratios,  choose a popular value of $j$.

This means, invoking once again the parameter $\eta=15/16$, that there is a popular set $S_j\subset TT$, for some $j$, with the following properties.  One has $\Omega(\eta^{j}  E)$ solutions of \eqref{teq}, corresponding to   $\Omega(2^{j-i} \eta^{j}  E/w)$ solutions  $(s,t_{ab},t_{cd},t_{ac},t_{bd}),$ for $s\in S_j$. Moreover, each $s\in S_j$ can be realised $\Omega(2^{i-j}w)$ times as a product in $TT$ (we do not claim any upper bound on the number of realisations of $s$, counting only those where $s=t_{ad}t_{bc}$). Changing the notation $S_j$ to $S$ and noting that $2^{i+j} \eta^{j} \geq 2^{\frac{1}{2}i+j}$ concludes the proof of  Lemma \ref{clue}. \end{proof}

We now conclude the proof of Theorem \ref{arr}.
Consider the set $S\subset TT$ coming from Lemma \ref{clue} and the equation \eqref{teqs} restricted to $s\in S$. The number of its solutions is
\begin{equation}\label{uppper}
\sum_{s\in S} r_{TT-TT}(s)\leq \sqrt{|S|} \sqrt{\sum_{s\in TT-TT}r^2_{TT-TT}(s)} = \sqrt{|S|Q}.
\end{equation}
where $r_{TT-TT}(s)$ is the number of representations of $s$ as an element of $TT-TT$. In the latter inequality we have used the Cauchy-Schwarz inequality.

By the claim of Lemma \ref{clue} on popularity of every $s\in S$ as a product in $TT$, we have 
$|S| = O(2^{j-i}w^{-1} |T|^2)$.

We substitute this to the latter bound \eqref{uppper} and compare what we get with the lower bound for the number of solutions of equation \eqref{teqs}, provided by Lemma \ref{clue}. This yields
\begin{equation}\label{pen}
|P_2|^{8/5}w^{7/5} 2^{\frac{1}{2}i+j}    \lesssim \sqrt{Q |T|^2/( 2^{i-j}w )} .
\end{equation}
 
The worst possible case of the estimate is $i=j=0$. (Note that in \eqref{is} we could, in fact, has a slightly higher power of $2$ than $2^{2i}$, hence an increasing factor in both $i$ and $j$ before the $\ll$ sign in the rearranged version of the rearranged estimate \eqref{pen}). This enables us to claim, using the upper bound \eqref{upper}, that
\begin{equation}\label{want}
 |T|^{17/2} \gtrsim |P_2|^{16/5} w^{19/5}.
\end{equation}

On the other hand, for the complement
$P_1$ of $P_2$ in $P$, one has, by Lemma \ref{both}:
 \[
 |\omega(P)| \geq |\omega(P_1)| \gg |P_1| w^{-1/2}.
 \]
 Optimising with \eqref{want} and $|\omega(P)|\geq |T|$ yields $
|\omega(P) |  \gtrsim  |P|^{\frac{108}{161}},$
 as claimed.  
\qed

\phantomsection

\addcontentsline{toc}{section}{References}

\bibliographystyle{plain}

\bibliography{RA_arxiv}

\newpage 

\noindent{B.~Murphy\\
University of Bristol,\\
Heilbronn Institute of Mathematical Research,\\
School of Mathematics,\\
University Walk, Bristol BS8 1TW, UK\\}
{\tt brendan.murphy@bristol.ac.uk}

\noindent{G.~Petridis\\
Department of Mathematics\\
University of Georgia\\
Athens, GA, 30602, USA\\}
{\tt giogis@cantab.net}

\noindent{O.~Roche-Newton\\
Johann Radon Institute for Computational and Applied Mathematics (RICAM),\\
69 Altenberger Stra{\ss}e,\\
Linz, Austria\\}
{\tt o.rochenewton@gmail.com}

\noindent{M.~Rudnev\\
University of Bristol,\\
School of Mathematics,\\
University Walk, Bristol BS8 1TW, UK\\}
{\tt m.rudnev@bristol.ac.uk}

\noindent{I.~D.~Shkredov\\
Steklov Mathematical Institute,\\
ul. Gubkina, 8, Moscow, Russia, 119991}
\\
and
\\
IITP RAS,  \\
Bolshoy Karetny per. 19, Moscow, Russia, 127994
\\
and
\\
MIPT, \\
Institutskii per. 9, Dolgoprudnii, Russia, 141701\\
{\tt ilya.shkredov@gmail.com}

\end{document}